\documentclass[reqno,11pt]{amsart}
\usepackage{amsmath,amssymb,amsthm,mathrsfs}
\usepackage{epsfig,color}
\usepackage[latin1]{inputenc}
\usepackage[english]{babel}
\usepackage[numbers]{natbib}

\usepackage{esint}

\usepackage{hyperref}

 \allowdisplaybreaks
\numberwithin{equation}{section}

\def\R{\mathbb R}
\def\N{\mathbb N}

\newcommand{\dist}{\mathop{\mathrm{dist}}}

\def\dist{{\rm dist}}

\def\00{{\bf 0}}

\newcommand{\tr}{\mbox{tr}\,}

\newcommand{\diver}{{\rm div \,}}

\newcommand{\meas}{{\rm meas \,}}

\newtheorem*{theorem*}{Theorem}
\newtheorem{theorem}{Theorem}[section]
\newtheorem{lemma}[theorem]{Lemma}
\newtheorem{proposition}[theorem]{Proposition}

\newtheorem{remark}[theorem]{Remark}

\newtheorem{definition}[theorem]{Definition}

\begin{document}
  
    \title[Symmetry for anisotropic $p$-Laplacian in convex cones]{Symmetry results for critical anisotropic \\
    $p$-Laplacian equations in convex cones}

  \date{}

\author{Giulio Ciraolo}
\address{G. Ciraolo. Dipartimento di Matematica e Informatica,
Universit\`a di Palermo, Via Archirafi 34, 90123 Palermo, Italy}
\email{giulio.ciraolo@unipa.it}

\author{Alessio Figalli}
\address{A. Figalli. ETH Z\"urich, Mathematics Department, R\"amistrasse 101, 8092 Z\"urich, Switzerland
}
\email{alessio.figalli@math.ethz.ch}

\author{Alberto Roncoroni} 
\address{A. Roncoroni. Dipartimento di Matematica ``F. Casorati'', Universit\`a di Pavia, Via Ferrata 5, 27100 Pavia, Italy}
\email{alberto.roncoroni01@universitadipavia.it}


\begin{abstract}
Given $n \geq 2$ and $1<p<n$, we
consider the critical $p$-Laplacian equation $\Delta_p u + u^{p^*-1}=0$, which corresponds to critical points of the Sobolev inequality.
Exploiting the moving planes method,
it has been recently shown that positive solutions in the whole space are classified.  
Since the moving plane method strongly relies on the symmetries of the equation and the domain,  
in this paper we provide a new approach to this Liouville-type problem that allows us to give a complete classification of solutions in an anisotropic setting.
More precisely, we characterize solutions to the critical $p$-Laplacian equation induced by a smooth norm inside any convex cone.
In addition, using optimal transport, we prove a general class of (weighted) anisotropic Sobolev inequalities inside arbitrary convex cones.
\end{abstract}

\maketitle


\section{Introduction}
Given $n \geq 2$ and $1<p<n$, we consider the critical $p$-Laplacian equation in $\mathbb{R}^n$, namely
\begin{equation} \label{eq0Euclidea}
\Delta_p u+ u^{p^*-1} = 0,
\end{equation}
where 
$$
p^*=\dfrac{np}{n-p}
$$
is the critical exponent for the Sobolev embedding.
The classification of positive solutions to \eqref{eq0Euclidea} in $\mathbb{R}^n$ started in the seminal papers \cite{GNN} and \cite{CGS} for $p=2$ and it has been the object of several studies. Recently, in \cite{Vetois} and \cite{Sc}, positive solutions to \eqref{eq0Euclidea} in $\mathbb{R}^n$ belonging to the class 
\begin{equation} \label{D1p}
\mathcal{D}^{1,p}(\mathbb{R}^n) := \left\{ u \in L^{p^*}(\mathbb{R}^n) \, :\ \nabla u \in L^{p}(\mathbb{R}^n) \right\} \,
\end{equation}
have been completely characterized. In particular, it is proved that a positive solution $u \in \mathcal{D}^{1,p}(\mathbb{R}^n)$ to \eqref{eq0Euclidea} must be of the form $u(x)= U_{\lambda,x_0} (x)$, where 
\begin{equation} \label{talentiane_p_intro_Euclidee}
U_{\lambda,x_0} (x) := \left( \frac{\lambda^{\frac{1}{p-1}}\left(n^{\frac{1}{p}} \left(\frac{n-p}{p-1}\right)^{\frac{p-1}{p}} \right)}{\lambda^\frac{p}{p-1} + |x-x_0|^\frac{p}{p-1} } \right)^{\frac{n-p}{p}} \,,
\end{equation}
for some $\lambda >0$ and $x_0 \in \mathbb{R}^n$. The approach used to achieve this classification needs a careful application of the method of moving planes, and it requires asymptotic estimates of $u$ and $\nabla u$ both from above and below.

When $p=2$ it is well-known that \eqref{eq0Euclidea} is related to Yamabe problem, and the classification result gives a complete classification of metrics on $\mathbb{R}^n$ which are conformal to the standard one (see \cite{Aubin,Schoen,Trudinger2,Yamabe} and the survey \cite{LeeParker}). 

For $1<p<n$, the study of solutions to \eqref{eq0Euclidea} is also related to critical points of the Sobolev inequality. Sobolev inequalities have been studied for more general norms as well as in convex cones (see \cite{BaerFig,CabreRosSerra,FigIndrei,FigMagPrat,LionsPacella,LionsPacellaTricarico}), where they take the form
\begin{equation} \label{sobolev_cones} 
\|u\|_{L^{p^*}(\Sigma)} \leq S_{\Sigma,H} \|H(\nabla u)\|_{L^{p}(\Sigma)} \,,
\end{equation}
where $H$ is a norm\footnote{By abuse of notation, we say that $H:\R^n\to \R$ is a norm if $H$ is convex, positively one-homogeneous (namely, $H(\ell \xi)=\ell H(\xi)$ for all $\ell>0$), and $H(\xi)>0$ for all $\xi \in \mathbb S^{n-1}$. Note that we do not require $H$ to be symmetric, so it may happen that $H(\xi)\neq H(-\xi)$.} and $\Sigma$ is a convex open cone in $\mathbb{R}^n$ given by
\begin{equation} \label{Sigma_def}
\Sigma=\lbrace tx \, : \, x\in\omega, \, t\in(0,+\infty)\rbrace
\end{equation}
for some open domain $\omega\subseteq \mathbb{S}^{n-1}$.

As far as we know, the sharp version of \eqref{sobolev_cones} is not available in literature and for this reason we provide a proof in Appendix \ref{appendix_sobolev} by suitably adapting the optimal transportation proof of the Sobolev inequality \cite{CNV} to the case of cones. It is interesting to observe that our proof applies also to the case of weighted Sobolev inequalities for the class of weights considered in \cite{CabreRosSerra}, thus generalizing \cite[Theorem 1.3]{CabreRosSerra} to the full range of exponents $p\in (1,n)$.

Hence, as shown in Appendix \ref{appendix_sobolev}, the extremals of \eqref{sobolev_cones} are of the form
\begin{equation} \label{talentiane_p_intro_H}
u(x)= U_{\lambda,x_0}^H (x) := \left( \frac{\lambda^{\frac{1}{p-1}}\left(n^{\frac{1}{p}} \left(\frac{n-p}{p-1}\right)^{\frac{p-1}{p}} \right)}{\lambda^\frac{p}{p-1} + H_0(x-x_0)^\frac{p}{p-1} } \right)^{\frac{n-p}{p}}
\end{equation}
for some $ \lambda >0 $  (see also \cite{Aubin_sobolev,CNV,LionsPacellaTricarico,Talenti} and the references therein), where $H_0$ denotes the dual norm associated to $H$, namely
$$
H_0(\zeta):=\sup_{H(\xi)=1}\zeta\cdot \xi\qquad \forall\,\zeta \in \R^n.
$$
Moreover, if $\Sigma = \mathbb{R}^n$ then $x_0$ may be any point of $\mathbb{R}^n$; if $\Sigma=\mathbb{R}^k\times\mathcal{C}$ with $k\in\{1,\dots,n-1\}$ and $\mathcal C$ does not contain a line, then $x_0\in\mathbb{R}^k\times\mathcal{\{\mathcal O\}}$; otherwise, $x_0=\mathcal{O}$
(from now on, $\mathcal O$ denotes the origin).

The aim of this paper is to provide a complete classification result for critical anisotropic $p$-Laplace equations in convex cones. More precisely, we consider the problem
\begin{equation}\label{p_laplace_cono}
\begin{cases}
\diver(a(\nabla u))  + u^{p^*-1} = 0 & \text{ in } \Sigma \\
u>0 & \text{ in } \Sigma \\
a(\nabla u) \cdot \nu =0 & \text{ on } \partial\Sigma  \\
u \in \mathcal{D}^{1,p}(\Sigma) \,,& 
\end{cases}
\end{equation}
where $\nu$ is the outward normal to $\partial \Sigma$,
\begin{equation} \label{a_def}
a(\xi)=H^{p-1}(\xi)\nabla H(\xi)\qquad \forall\,\xi \in \R^n,
\end{equation}
and the space $\mathcal{D}^{1,p}(\Sigma)$
is defined as in \eqref{D1p} (with $\R^n$ replaced by $\Sigma$).
We will sometimes write 
\begin{equation*}
\Delta_p^H u=\diver(a(\nabla u))\, ,
\end{equation*}
where $\Delta_p^H$ is called the \emph{Finsler p-Laplacian} (or \emph{anisotropic p-Laplacian}) operator. It is clear that when we consider the case $\Sigma = \mathbb{R}^n$ no boundary conditions are given. 

We observe that if $u\in \mathcal{D}^{1,p}(\Sigma)$ is a positive critical point for the Sobolev functional
\begin{equation}\label{sobolev_functional_anisotr_coni}
J(u)=\dfrac{\int_\Sigma H(\nabla u)^p dx}{\left(\int_\Sigma |u|^{p^*} dx \right)^\frac{p}{p^*}},
\end{equation}
then $u$ satisfies \eqref{p_laplace_cono}.
The main goal of this paper is to classify the critical points for \eqref{sobolev_functional_anisotr_coni}, i.e. the classification of the solutions to \eqref{p_laplace_cono}.

\begin{theorem} \label{thm_main}
Let $n\geq 2$, $1<p<n$ , and let $\Sigma=\mathbb{R}^k\times\mathcal{C}$ be a convex cone, where $\mathcal C$ does not contain a line. Let $H$ be a norm of $\mathbb R^n$ such that $H^2$ is of class $C^{2}(\R^n\setminus \{\mathcal O\})$ and it is uniformly convex and $C^{1,1}$ in $\R^n$, namely there exist constants $0<\lambda \leq \Lambda$ such that
\begin{equation}
\label{eq:elliptic H}
\lambda {\rm Id}\leq H(\xi)\,D^2H(\xi)+\nabla H(\xi)\otimes \nabla H(\xi) \leq \Lambda\,{\rm Id}\qquad \forall\,\xi \in \R^n\setminus \{\mathcal O\}
\end{equation}
(note that $D^2(H^2)=2H\,D^2H+2\nabla H\otimes \nabla H$).

Let $u$ be a solution to \eqref{p_laplace_cono}.
Then $u(x)=U_{\lambda,x_0}^H (x)$ for some $ \lambda >0 $ and $ x_0 \in \overline{\Sigma}$, where $U_{\lambda,x_0}^H$ is given by \eqref{talentiane_p_intro_H}.
Moreover, 
\begin{itemize}
	\item[$(i)$] if $k=n$ then $\Sigma = \mathbb{R}^n$ and $x_0$ may be a generic point in $\mathbb{R}^n$;
	\item[$(ii)$] if $k\in\{1,\dots,n-1\}$ then $x_0\in\mathbb{R}^k\times\mathcal{\{\mathcal O\}}$;
	\item[$(iii)$] if $k=0$ then $x_0=\mathcal{O}$.
\end{itemize} 
\end{theorem}

As already mentioned, case $(i)$ in Theorem \ref{thm_main} has been already proved in \cite{CGS,DamascelliSciunzi, Sc,Vetois} when $\Sigma=\mathbb{R}^n$ and $H$ is the Euclidean norm. In that case, thanks to the symmetry of the problem, the  authors can apply the method of moving planes. In the Euclidean case and for $p=2$, the classification of solutions in convex cones was proved in \cite[Theorem 2.4]{LionsPacellaTricarico} by using the Kelvin transform and inspired by \cite{Gidas}.   Unfortunately, the Kelvin transform and the method of moving planes are not helpful neither for anisotropic problems nor inside cones for a general $p \in (1,n)$. 
For this reason we provide a new approach to the characterization of solutions to critical $p-$Laplacian equations, which is based on integral identities rather than moving planes.
 This approach takes inspiration from \cite{BiCi,BCS,BNST} where classical overdetermined problems for PDEs are considered (see also \cite{CirRon,PacellaTralli} for analogous problems in convex cones).

 \subsection*{Strategy of the proof and structure of the paper} The strategy of the proof can be explained as follows. First, using that $u \in \mathcal{D}^{1,p}(\Sigma)$ we show that $u$ is bounded (see Subsection \ref{subsect_boundness}).
 Then, in Subsection \ref{subsect_asymp} we prove that $u$  satisfies certain decay estimates at infinity (in particular it behaves as the fundamental solution both from above and below), so that one has optimal upper bounds on $H(\nabla u)$ in terms of the fundamental solution.  We notice that, differently from \cite{Sc}, we do not need asymptotic lower bounds on $\nabla u$; instead, we use a Caccioppoli-type inequality to prove some asymptotic estimates on certain integrals involving higher order derivatives (see Subsection \ref{subsect_caccioppoli}). 

Then, in Section \ref{sect_proof} we consider the auxiliary function $v= u^{-\frac{p}{n-p}}$. We find the  elliptic equation satisfied by $v$ and then, thanks to the asymptotic estimates on $u$, we show that $v$ and $\nabla v$ satisfy explicit growth conditions at infinity. By using integral identities, the convexity of $\Sigma$, and some suitable inequalities, we are able to prove that $\nabla a(\nabla v)$ is a multiple of the identity matrix, from which the symmetry result follows.

In Appendix \ref{appendix_sobolev} we prove the sharp version of \eqref{sobolev_cones} for general norms and cones, and even in a weighted setting.

Most of the paper will focus on the case in which $\Sigma$ is a convex cone with nonempty boundary. Indeed our approach perfectly works also when $\Sigma = \mathbb{R}^n$. However, since the whole space case is simpler to be proven, we prefer to focus the exposition to the case when $\Sigma$ has  boundary.

\subsection*{Acknowledgments} The authors wish to thank Andrea Cianchi and Alberto Farina for useful discussions. G.C. and A.R. have been partially supported by the ``Gruppo Nazionale per l'Analisi Matematica, la Probabilit\`a e le loro Applicazioni'' (GNAMPA) of the ``Istituto Nazionale di Alta Matematica'' (INdAM, Italy). G.C. has been partially supported by the PRIN 2017 project ``Qualitative and quantitative aspects of nonlinear PDEs''. A.F. has been partially supported by European Research Council under the Grant
Agreement No 721675. Part of this manuscript was written while A.R. was visiting the Department of Mathematics of the ETH in Z\"{u}rich, which is acknowledged for the hospitality.

\section{Preliminary results}
In this section we collect some results that are well established when $\Sigma = \mathbb{R}^n$ and $H$ is the Euclidean norm. Since we are dealing with problem \eqref{p_laplace_cono} and some modifications are needed, we report here their counterpart when $\Sigma$ is a convex cone and $H$ a general norm, and provide a sketch of the proofs emphasizing the main differences.

In the whole paper we denote by $B_r(x)$ the usual Euclidean ball, and by $B_r$ the ball $B_r(\mathcal O)$ centered at the origin.

\subsection{Boundeness of solutions} \label{subsect_boundness}
In the following lemma we prove that solutions to \eqref{p_laplace_cono} are bounded. The result holds for more general Neumann problems, in particular for problems with a differential operator modelled on the $p$-Laplace operator.

\begin{lemma}\label{Lemma_Cianchi_prel}
	Let $\Sigma\subseteq \mathbb{R}^n$ be a convex cone as in \eqref{Sigma_def} and let $u\in \mathcal{D}^{1,p}(\Sigma)$ be a solution to 
	\begin{equation}\label{approx_ter}
	\begin{cases}
	\diver(a(\nabla u)) + u^{p^*-1} = 0 & \text{ in } \Sigma \\
	u>0 & \text{ in } \Sigma \\
	a(\nabla u)\cdot\nu=0 & \text{ on } \partial\Sigma\, ,
	\end{cases}
	\end{equation}
	where the $a:\mathbb{R}^n\rightarrow\mathbb{R}^n$ is a continuous vector field such that the following holds: 
	there exist $\alpha>0$ and $0\leq s\leq 1/2$ such that
	\begin{equation}\label{ellipticity_approx}
	|a(\xi)|\leq \alpha(|\xi|^2+s^2)^{\frac{p-1}{2}} \quad \text{and} \quad \xi\cdot a(\xi)\geq \dfrac{1}{\alpha}\int_0^1\left( t^2|\xi|^2+s^2\right)^{\frac{p-2}{2}}|\xi|^2\, dt  \, , 
	\end{equation}
	for every $\xi\in\mathbb{R}^n$. Then there exists $\delta>0$ with the following property: let $\rho>0$ be such that 
	$$
	||u||_{L^{p^\ast}(B_\rho(x_0))}\leq \delta\qquad \forall\,x_0 \in \mathbb R^n.
	$$
Then 
	$$
	||u||_{L^\infty(\Sigma\cap B_{R/2}(x_0))}\leq CR^{-\frac{n}{p}}||u||_{L^p(\Sigma\cap B_R(x_0))}\qquad \forall\,R\leq \rho,
	$$ 
where $C$ depends only on $n$, $\alpha$, $p$ and the Sobolev constant of $\Sigma$.
\end{lemma}

\begin{proof} We closely follow \cite[Theorem E.0.20]{Peral} and \cite[Theorem 1]{Serrin_local} and we only give a sketch of the proof. We first prove that  $u\in L^{q p^\ast}_{\rm loc}(\overline\Sigma)$  for any $q <p^\ast/p.$ Given $l>0$ and $1<q<\frac{p^{\ast}}{p}$, we define
	\begin{equation}\label{F_Peral}
	F(u)=\begin{cases}
	u^{q}\ & \text{ if } u\leq l \\
	q l^{q-1}(u-l)+l^q & \text{ if } u>l  \, ,
	\end{cases}
	\end{equation}
	and 
	$$
	G(u)=\begin{cases}
	u^{(q-1)p+1}\ & \text{ if } u\leq l \\
	((q-1)p+1)l^{(q-1)p}(u-l)+ l^{(q-1)p+1} & \text{ if } u>l  \,.
	\end{cases}
	$$
	Let $\eta\in C^\infty_0(\mathbb{R}^n)$ and  use
	$$
	\xi=\eta^p G(u)
	$$
	as a test-function in \eqref{approx_ter}; then an integration by parts gives
	\begin{equation}\label{debole}
	\int_{\Sigma} {a}(\nabla u)\cdot\nabla (\eta^p G(u))\, dx=\int_{\Sigma}u^{p^{\ast}-1}\eta^p G(u)\, dx\, .
	\end{equation}
	We aim at proving that 
	\begin{equation}\label{17.30}
\begin{aligned}
c\int_{\Sigma} \eta^p G'(u)|\nabla u|^p\, dx \leq & \int_{\Sigma}\eta^{p-1}G(u)|a(\nabla u)\cdot\nabla\eta|\, dx
+\int_{\Sigma}u^{p^{\ast}-1}\eta^p G(u)\, dx \\
&+s^{p}\int_\Sigma\eta^p G'(u)\, dx 
\end{aligned}
\end{equation}
	holds for $0\leq s\leq 1/2$. We distinguish between the cases $1<p<2$ and $2\leq p<n$. 
	
	If $p\geq 2$, then \eqref{ellipticity_approx} implies
	$$
	\xi\cdot a(\xi)\geq \dfrac{1}{\alpha}|\xi|^p\,,
	$$
	and from \eqref{debole} we get
	\begin{equation*}
	\dfrac{1}{\alpha}\int_{\Sigma} \eta^p G'(u)|\nabla u|^p\, dx \leq p\int_{\Sigma}\eta^{p-1}G(u)|a(\nabla u)\cdot\nabla\eta|\, dx
	+\int_{\Sigma}u^{p^{\ast}-1}\eta^p G(u)\, dx \, ,
	\end{equation*}
	which implies \eqref{17.30}.
	
	If $1<p<2$ then \eqref{17.30} is obtained by using a more careful argument. 
	We claim that
	\begin{equation}\label{eq:claim}
	\int_0^1\left( t^2|\xi|^2+s^2\right)^{\frac{p-2}{2}}|\xi|^2\, dt \geq \frac12\left(|\xi|^p - s^p\right).
	\end{equation}
	To prove this we consider two cases. If $s > |\xi|$ then the left-hand side of \eqref{eq:claim} is negative, and so the result is clearly true. Otherwise, if $s \leq |\xi|$ then
	$$
	 t^2|\xi|^2+s^2 \leq 2|\xi|^2 \qquad \text{ for $t \in [0,1]$}, 
	$$
	and therefore
	$$
	\int_0^1\left( t^2|\xi|^2+s^2\right)^{\frac{p-2}{2}}|\xi|^2\, dt \geq \int_{0}^1\left( 2|\xi|^2\right)^{\frac{p-2}{2}}|\xi|^2\, dt = 2^{\frac{p-2}2}|\xi|^p \ge \frac12 |\xi|^p,
	$$
	that again implies \eqref{eq:claim}.
	
Thanks to \eqref{debole}, \eqref{ellipticity_approx}, and \eqref{eq:claim}, we obtain
	\begin{equation*}
	\begin{aligned}
	\dfrac{1}{2\alpha}\int_{\Sigma} \eta^p G'(u)|\nabla u|^p\, dx &\leq  p\int_{\Sigma}\eta^{p-1}G(u)|a(\nabla u)\cdot\nabla\eta|\, dx
	+\int_{\Sigma}u^{p^{\ast}-1}\eta^p G(u)\, dx \\
	&+\frac{s^p}{2}\int_\Sigma\eta^p G'(u)\, dx \, ,
	\end{aligned}
	\end{equation*}
	and the proof of \eqref{17.30} is complete.

	Note now that, by Young's inequality and  \eqref{ellipticity_approx}, for any $\epsilon\in (0,1)$ we have
	\begin{equation*}
	\begin{aligned}
	\eta^{p-1} |a(\nabla u)\cdot\nabla\eta|&\leq
	\epsilon^{\frac{p}{p-1}}u^{-1}|a(\nabla u)|^{\frac{p}{p-1}}\eta^p+\epsilon^{-p}u^{p-1}|\nabla\eta|^p\\
	&\leq C_0\epsilon^{\frac{p}{p-1}}u^{-1}(|\nabla u|^p+s^p)\eta^p+\epsilon^{-p}u^{p-1}|\nabla\eta|^p,
	\end{aligned}
	\end{equation*}
	where $C_0$ depends only on $\alpha$ and $p$.
	Thanks to this inequality and recalling \eqref{17.30}, since $G(u)\leq uG'(u)$ (note that $G$ is convex and $G(0)=0$), for any $\epsilon\in (0,1)$ we obtain 
	\begin{equation*}
	\begin{aligned}
	c\int_{\Sigma} \eta^p G'(u)|\nabla u|^p\, dx &\leq C_0\epsilon^{\frac{p}{p-1}} \int_{\Sigma} \eta^p G'(u)|\nabla u|^p\, dx
	+ (C_0+1) s^p\int_{\Sigma} \eta^p G'(u)\, dx \\
	&+\epsilon^{-p}\int_{\Sigma}G(u)u^{p-1}|\nabla\eta|^p\, dx+ \int_{\Sigma}u^{p^{\ast}-1}\eta^p G(u)\, dx\, .
	\end{aligned}
	\end{equation*}
	Hence, choosing $\epsilon$ small enough so that $C_0\epsilon^{\frac{p}{p-1}}=c/2$, we deduce that
	$$
	c'\int_{\Sigma} \eta^p G'(u)|\nabla u|^p\, dx \leq s^p\int_{\Sigma} \eta^p G'(u)\, dx \\
	+\int_{\Sigma}G(u)u^{p-1}|\nabla\eta|^p\, dx+ \int_{\Sigma}u^{p^{\ast}-1}\eta^p G(u)\, dx\, ,
	$$
	where $c'>0$ depends only on $n$, $\alpha$, and $p$.
	Using now that $G'(u)\geq c[F']^p$ and that $u^{p-1}G(u)\leq C[F(u)]^p$, we obtain  
	\begin{equation*}
	\hat c\int_{\Sigma}|\nabla (\eta F(u))|^p\, dx \\ \leq s^{p}\int_{\Sigma}\eta^pG'(u)\, dx+\int_{\Sigma}|\nabla\eta|^pF^p(u)\, dx+ \int_{\Sigma}\eta^pu^{p^{\ast}-p} F^p(u)\, dx\, .
	\end{equation*}
Hence, thanks to the Sobolev inequality \eqref{sobolev_cones}  we get
	\begin{equation}\label{post_sobolev}
	\bar c\left(\int_{\Sigma}F^{p^\ast}(u)\eta^{p^\ast}\, dx\right)^{\frac{p}{p^\ast}}  \leq s^{p}\int_{\Sigma}\eta^pG'(u)\, dx+\int_{\Sigma}|\nabla\eta|^pF^p(u)\, dx + \int_{\Sigma}\eta^pu^{p^{\ast}-p} F^p(u)\, dx \, ,
	\end{equation}
	where $\bar c>0$ depends only on $n$, $\alpha$, $p$ and the Sobolev constant for $\Sigma$.
	
	Now, choose $\delta=(\bar c/2)^{1/(p^\ast - p)}$, so that for any  $R \leq \rho$ it holds
	$$
	||u||^{p^\ast-p}_{L^{p^\ast}(B_R(x_0))}\leq \frac{\bar c}2\qquad \forall\,x_0\in \R^n.
	$$
Then, if we choose $\eta$ such that ${\rm supp}(\eta)\subset B(x_0, R)$, it follows from Holder's inequality that we can reabsorb the last term in \eqref{post_sobolev}, and we get
	$$
	\frac{\bar c}2\left(\int_{\Sigma}F^{p^\ast}(u)\eta^{p^\ast}\, dx\right)^{\frac{p}{p^\ast}}\leq s^{p}\int_{\Sigma\cap B_R(x_0)}\eta^pG'(u)\, dx+\int_{\Sigma\cap B_R(x_0)}|\nabla \eta|^pF^p(u)\, dx\, .
	$$
	Hence, taking the limit as $l\rightarrow\infty$ in the definition of $F$ and $G$, by monotone convergence we conclude
	\begin{equation*}
	\frac{\bar c}2\left(\int_{\Sigma\cap B_R(x_0)}\eta^{p^\ast}u^{q p^\ast}\, dx\right)^{\frac{p}{p^\ast}}\leq 
	s^{p}\int_{\Sigma\cap B_R(x_0)}u^{(q-1)p}\, dx+ ||\nabla \eta||^p_{\infty}\int_{\Sigma\cap B_R(x_0)}u^{q p}\, dx\, .
	\end{equation*}
	Since $qp<p^{\ast}$ it follows that the right hand side is finite, hence by the inequality above and the arbitrariness of $x_0$ we conclude that $u\in L^{q p^\ast}_{\rm loc}(\overline\Sigma)$.

Thanks to this information, we can rewrite the equation satisfied by $u$ as follows:
	$$
	-\diver(a(\nabla u))=f(x)u^{p-1}+g(x)
	$$ 
	where 
	$$
	f(x)=\begin{cases}
	0 & \text{ if } u< 1 \\
	u^{p^\ast-p} & \text{ if } u\geq 1  \, ,
	\end{cases}
	$$
	and 
	$$
	g(x)=\begin{cases}
	0 & \text{ if } u> 1 \\
	u^{p^\ast-1} & \text{ if } u\leq 1  \, .
	\end{cases}
	$$
	Since $u\in L^{q p^\ast}_{\rm loc}$ we get that $f\in L^r$ with $r>\frac{n}{p}$ and $g\in L^{\infty}$.
	Hence, as in the proof of  \cite[Theorem 1]{Serrin_local},
	a classical  Moser iteration argument yields the result.
\end{proof}

\begin{remark}
\label{rmk:Holder}{\rm 
As observed in the proof of \cite[Theorem E.0.20]{Peral}, the Moser iteration argument can also be used to show that $u$ is uniformly $C^{0,\theta}$ up to the boundary.}
\end{remark}
\subsection{Asymptotic bounds on $u$ and $\nabla u$} \label{subsect_asymp}

The main goal of this subsection is to prove Proposition \ref{Vetosi_anis} below. Proposition \ref{Vetosi_anis} is a generalization of \cite[Theorem 1.1]{Vetois} to the conical-anisotropic setting. The proof of Proposition \ref{Vetosi_anis} follows the one given in \cite{Vetois}, although the lack of smoothness of $\Sigma$ creates some nontrivial extra difficulties. 

\begin{proposition}\label{Vetosi_anis}
	Let $1<p<n$ and let $u$ be a solution to \eqref{p_laplace_cono}. Then there exist two positive constants $C_0$ and $C_1$ such that
	\begin{equation}\label{1.3-1.4_anis}
	\dfrac{C_0}{1+|x|^{\frac{n-p}{p-1}}}\leq u(x)\leq \dfrac{C_1}{1+|x|^{\frac{n-p}{p-1}}} \quad \text{and} \quad |\nabla u(x)|\leq \dfrac{C_1}{1+|x|^{\frac{n-1}{p-1}}}\, ,
	\end{equation}
	for all $x\in \Sigma$.
\end{proposition}

Before giving the proof of Proposition \ref{Vetosi_anis}, we first introduce a useful definition.

\begin{definition}
Given $L>0$, we say that a convex cone $\mathcal C$ is $L$-Lipschitz if for any point $x \in \partial \mathcal C$ there exist $r_x>0$ and a unit vector $\nu_x$ such that
$$
B_{r_x}(x+Lr_x \nu_x)\subset \mathcal C.
$$
Note that, by convexity of $\mathcal C$, also the convex hull of $B_{r_x}(x+L\nu_x)\cup\{x\}$ is contained in $\mathcal C$.
\end{definition}

In the spirit of \cite[Lemma 2.3]{Vetois}, we now prove a general lower bound on the $L^{p^*}$ norms of solutions to our equation in convex cones, with a bound depending only on the Lipschitz constant (see also \cite{LionsPacellaTricarico}).

\begin{lemma}[Lower bound on the mass] \label{lem:lower bound} Let $u$ be a nontrivial solution to 
\begin{equation}\label{p_laplace_cono_i}
\begin{cases}
\diver(a(\nabla u))  + u^{p^*-1} = 0 & \text{ in } \mathcal C \\
u>0 & \text{ in } \mathcal C \\
a(\nabla u) \cdot \nu =0 & \text{ on } \partial\mathcal C  \\
u \in \mathcal{D}^{1,p}(\mathcal C) \,,& 
\end{cases}
\end{equation}
where $\mathcal C$ is a $L$-Lipschitz convex cone and $a(\xi)$ is as in \eqref{a_def}.
Then there exists a constant $k_0>0$, depending only on $n$, $p$, $L$, and $\min_{\mathbb S^{n-1}}H$, such that 
$$
\|u\|_{L^{p^*}(\mathcal C)}\geq k_0.
$$
\end{lemma}
\begin{proof}
As in \cite[Lemma 2.3]{Vetois}, the proof is based on the Sobolev inequality in $\mathcal C$, and
on the integral identity that one obtains by multiplying \eqref{p_laplace_cono_i} by $u$ and integrating in $\mathcal C$.
However in this case a bit more carefulness is needed, especially to quantify the dependencies.

First of all, up to a translation, we can assume that $\mathcal C$ has vertex at $O$.
Then, since $\mathcal C$ is $L$-Lipschitz, there exist $r_0>0$ and a unit vector $\nu_0$ such that
$B_{r_0}(Lr_0\nu_0)\subset \mathcal C.$
Therefore, since $\mathcal C$ is a convex cone, this implies that the cone
$$
\hat {\mathcal C}_L:=\bigcup_{r>0}B_r(Lr\nu_0)
$$
is contained inside $\mathcal C$.

We now want to estimate the Sobolev constant of $\mathcal C$.
To this aim we define the following constant:
$$
\mathcal S_L:=\inf \Biggl\{\frac{\left(\int_{\Omega} |\nabla \varphi|^p dx\right)^{1/p}}{\left( \int_\Omega |\varphi|^{p^*}dx\right)^{1/p^*}}\,:\, \text{$\Omega$ is convex, $B_1\cap \hat {\mathcal C}_L \subset \Omega\subset B_1$, $\varphi \in C^1(\Omega)$, $\varphi|_{\partial B_1\cap \hat {\mathcal C}_L}=0$} \Biggr\}.
$$
Since the set of convex domains $\Omega\subset B_1$ containing $B_1 \cap \hat {\mathcal C}_L$ are uniformly Lipschitz, standard arguments in the calculus of variations show that $\mathcal S_L$ is positive.

We now notice that, given any function $\psi \in C^1_c(\overline{\mathcal C})$, there exists $\lambda>0$ large such that $\psi_\lambda(x):=\psi(\lambda x)$ satisfies
$\psi_\lambda \in C^1(\mathcal C)$ and $\psi_\lambda|_{\partial B_1\cap \hat {\mathcal C}_L}=0$ (since $\partial B_1\cap \hat {\mathcal C}_L\subset \partial B_1\cap {\mathcal C}$). Hence, we can bound
$$
\frac{\left(\int_{\mathcal C} |\nabla \psi|^p dx\right)^{1/p}}{\left( \int_{\mathcal C} |\psi|^{p^*}dx\right)^{1/p^*}}=\frac{\left(\int_{\mathcal C} |\nabla \psi_\lambda|^pdx \right)^{1/p}}{\left( \int_{\mathcal C} |\psi_\lambda|^{p^*}dx\right)^{1/p^*}}\geq \mathcal S_L.
$$
Since $\psi \in C^1_c(\overline{\mathcal C})$ is arbitrary, it follows by approximation that
$$
\left(\int_{\mathcal C} |\nabla \psi|^pdx \right)^{1/p} \geq \mathcal S_L \left( \int_{\mathcal C} |\psi|^{p^*}dx\right)^{1/p^*} \qquad \forall\,\psi \in \mathcal{D}^{1,p}(\mathcal C).
$$
Applying this inequality to $u$ and defining $c_H:=\min_{|\xi|=1}H(\xi)$, we get
$$
\int_{\mathcal C}H(\nabla u)^p dx \geq c_H^p\int_{\mathcal C}|\nabla u|^p dx\geq (c_H\mathcal S_L)^p\left( \int_{\mathcal C} u^{p^*}dx\right)^{p/p^*}.
$$
On the other hand, multiplying \eqref{p_laplace_cono_i} by $u$ and integrating in $\mathcal C$, we get
$$
\int_{\mathcal C}H(\nabla u)^p dx=\int_{\mathcal C} u^{p^*}dx.
$$
Combining the last two equations yield the desired lower bound.
\end{proof}

\begin{remark}
{\rm 
An alternative proof of Lemma \ref{lem:lower bound} can be obtained by computing the optimal Sobolev constant of $\mathcal C$ (using Appendix A) and noticing that this constant is bounded below in terms only of $n$, $p$, $H_0$, and the volume of $\mathcal C\cap B_1$. In particular, whenever $\mathcal C$ is $L$-Lipschitz then $\hat{\mathcal C}_L \subseteq \mathcal C $ and 
$|\mathcal C\cap B_1|\geq |\hat{\mathcal C}_L\cap B_1|$, and one concludes that the Sobolev constant of $\mathcal C$ is controlled by (actually, it is larger or equal than) the one of $\hat{\mathcal C}_L$.}
\end{remark}

We shall also need  a doubling-type property on $u$
 which is proved in \cite[Lemma 5.1]{PQS} (see also \cite[Lemma 3.1]{Vetois}). Below we state a version of this doubling property which is suitable for our setting.
 
 Note that, by convexity, there exists a constant $L_\Sigma>0$ such that $\Sigma$ is $L_\Sigma$-Lipschitz. Then we let $k_0>0$ be the constant provided by Lemma \ref{lem:lower bound} with $L=L_\Sigma$.

\begin{lemma}[Doubling property \cite{PQS}] \label{doubling_anis} Let $u$ be a solution to \eqref{p_laplace_cono_i},
let $L_\Sigma$ be the Lipschitz constant of $\Sigma$, and let $k_0>0$ be the constant provided by Lemma \ref{lem:lower bound} with $L=L_\Sigma$.

 Let $k\in(0,k_0)$, $r>0$, and $r'\in(0,r)$ be fixed, and set 
		$$
		r''=\dfrac{r+r'}{2}\, .
		$$	
		Then for any $x\in\overline{\Sigma}\setminus B_{r''}$ and $\alpha>0$ such that the distance $d$ between $x$ and $\Sigma\cap B_{r''}$ satisfies
		\begin{equation}
		d(x,\Sigma\cap B_{r''})u(x)^{\frac{p}{n-p}}>2\alpha\, ,
		\end{equation}
		there exists a point $y_0\in\overline{\Sigma}\setminus B_{r''}$ such that
		\begin{equation}
		d(y_0,\Sigma\cap B_{r''})u(x)^{\frac{p}{n-p}}>2\alpha\, , \quad u(x_0)\leq u(y_0)\, ,
		\end{equation}
		and 
		\begin{equation}
		u(y)\leq 2^{\frac{n-p}{p}}u(y_0)\quad \text{for all } \, y\in \Sigma\cap B_{\bar r}(y_0)\, ,
		\end{equation}
		where $\bar r=\alpha u(y_0)^{-\frac{p}{n-p}}$.
	\end{lemma}

\begin{proof}[Proof of Proposition \ref{Vetosi_anis}]
	We divide the proof of Proposition \ref{Vetosi_anis} in three steps. In Step 1 we give a preliminary decay estimate on $u$ (which is not sharp). In Step 2 we prove that $u \in L^{\hat p-1,\infty}(\Sigma)$ for a suitable $\hat{p}$. Finally, in Step 3 we prove \eqref{1.3-1.4_anis}.

	\medskip

\noindent $\bullet$
\emph{Step 1:} {\it Let $u$ be a solution of \eqref{p_laplace_cono}, and for $k \in (0,k_0)$  define 
		\begin{equation}
		r_k(u):=\inf\lbrace r>0\, :\, ||u||_{L^{p^{\ast}}(\Sigma\setminus B_{r})}<k\rbrace \,.
		\end{equation}
		Then, for any fixed $k\in(0,k_0)$ and $r>r_k(u)$, there exists a constant $K_0$ such that 
		\begin{equation}\label{lorenzino}
		|u(x)|\leq K_0 H_0(x)^{\frac{p-n}{p}} \quad \textit{ for all }\, x\in \overline{\Sigma}\setminus B_{r}\, .
		\end{equation}
	}

In order to prove the assertion, it suffices to show the existence of a constant $K_1$ such that 
	\begin{equation}\label{by_contradiction_anis}
	d(x,\Sigma\cap B_{r''})u(x)^{\frac{p}{n-p}}\leq K_1 \quad \text{for all } x\in\overline{\Sigma}\setminus B_{r}\,,
	\end{equation}
	where $r''=(r+r')/2$ and $r'\in (0,r)$ is fixed.
	We prove \eqref{by_contradiction_anis} by contradiction.
	
	Suppose there exists a sequence of points $\lbrace x_\alpha\rbrace_{\alpha\in\mathbb{N}}\subset \overline{\Sigma}\setminus B_{r}$ such that
	\begin{equation}\label{3.5_anis}
	d(x_\alpha,\Sigma\cap B_{r''})u(x_\alpha)^{\frac{p}{n-p}}> 2\alpha\, . 
	\end{equation}
	Since $B_{r''}\subset B_{r}$, it follows from \eqref{3.5_anis} and Lemma \ref{doubling_anis} that there exists a sequence of points $\lbrace y_\alpha\rbrace_{\alpha\in\mathbb{N}}\subset \overline{\Sigma}\setminus B_{r''}$ such that
	\begin{equation}\label{3.6_anis}
	d(y_\alpha,\Sigma\cap B_{r''})u(y_\alpha)^{\frac{p}{n-p}}> 2\alpha\, , \quad u(x_\alpha)\leq u(y_\alpha)\, ,
	\end{equation}
	and 
	\begin{equation}\label{3.7_anis}
	u(y)\leq 2^{\frac{n-p}{p}}u(y_\alpha) \quad \text{for all } \, y\in \Sigma\cap B_{\bar r}(y_\alpha) \, .
	\end{equation}
	We observe that, since  $u$ is bounded, the sequences $\{ x_\alpha\}_{\alpha\in\mathbb{N}}$ and $\{ y_\alpha\}_{\alpha\in\mathbb{N}}$ are both divergent as $\alpha \to \infty$.

	For any $\alpha\in\mathbb{N}$ and $y\in\overline{\Sigma}$, we define
	\begin{equation}\label{3.8_anis}
	\tilde{u}_\alpha(y):=u(y_\alpha)^{-1} u(m_\alpha^{-1} y+y_\alpha)\, 
	\end{equation}
	where $m_\alpha:=u(y_\alpha)^{\frac{-p}{n-p}}$.
	From \eqref{p_laplace_cono} we obtain 
	\begin{equation}\label{eq_tilde_u_vera}
	\begin{cases}
	-\Delta^H_p \tilde{u}_\alpha= \tilde{u}_\alpha^{p^\ast-1} & \text{ in } \Sigma_\alpha \\
	\tilde u_\alpha(\mathcal O)=1,\\
	a(\nabla \tilde{u}_\alpha)\cdot \nu=0 & \text{ on } \partial\Sigma_\alpha\, ,
	\end{cases}
	\end{equation}
	where 
	$$
	\Sigma_\alpha:=m_\alpha(\Sigma-y_\alpha)=\{y \in \R^n\,:\,m_\alpha^{-1} y+y_\alpha \in \Sigma\}
	$$
	is a convex cone.
	
	It is immediate to check that  the cones $\Sigma_\alpha$
	are $L_\Sigma$-Lipschitz.
	Furthermore, if we set $\mu_\alpha:=u(y_\alpha)^{-1}$, \eqref{3.7_anis} and \eqref{3.8_anis} yield that 
	\begin{equation}\label{3.11_vera}
	\tilde{u}_\alpha(-y_\alpha m_\alpha)=\mu_\alpha u(\mathcal O)\neq 0 \ \  \text{ and } \ \  \tilde{u}_\alpha(y)\leq 2^{\frac{n-p}{p}} \quad \text{ for all } \, y\in \Sigma_\alpha\cap B_{\alpha}\,.
	\end{equation}
	At this point we consider the ratio 
	$$
	q_\alpha:=\dfrac{m_\alpha}{|y_\alpha|}\,. 
	$$
Observe that (by \eqref{3.6_anis})  $q_\alpha\rightarrow 0$ as $\alpha\rightarrow\infty$. 
	
	Since $|y_\alpha| \to + \infty $, the ratio between $-y_\alpha m_\alpha$  and the scaling factor $ m_\alpha$ goes to infinity.
	Hence, one of the following two cases may occur  as $\alpha\rightarrow\infty$ : 
	\begin{itemize}
		\item[$(i)$] the sequence of cones $\{\Sigma_\alpha\}_{\alpha\in\mathbb{N}}$ converges to $\mathbb{R}^n$ (this happens if the distance between $m_\alpha y_\alpha$ and $\partial \Sigma_\alpha$ goes to infinity);
		\item[$(ii)$] the sequence of cones $\{\Sigma_\alpha\}_{\alpha\in\mathbb{N}}$ converges to a $L_\Sigma$-Lipschitz convex cone $\mathcal C$, not necessarily centered at the origin (this happens if the distance between $m_\alpha y_\alpha$ and $\partial \Sigma_\alpha$ remains bounded).
	\end{itemize}
We now look in both cases at the behavior of the functions  $\{u_\alpha\}_{\alpha\in\mathbb{N}}$. We consider the two cases separately.
	
- Case $(i)$: fix  a ball $B_{R}$. Then there exists $\overline{\alpha}\in\mathbb{N}$ such that 
	$\Sigma_\alpha\cap B_{R}=B_{R}$ for every $\alpha\geq\overline{\alpha}$; moreover $\tilde{u}_\alpha$ (for every $\alpha\geq\overline{\alpha}$) is a solution of \eqref{eq_tilde_u_vera} in $B_{R}$.
	From \eqref{eq:elliptic H}, \eqref{3.11_vera}, and \cite{DiBenedetto}, there exist a constant $C>0$ and a real number $\theta\in(0,1)$ such that
	\begin{equation}\label{3.12_vera}
	||\tilde{u}_\alpha||_{C^{1,\theta}(B_{R/2})}\leq C
	\end{equation}
	for any $\alpha\geq\overline{\alpha}$. Since $R>0$ is arbitrary,  Ascoli-Arzel\`a Theorem 
	and a diagonal argument imply that $\{\tilde{u}_\alpha\}_{\alpha\in\mathbb{N}}$ converges (up to subsequence) in $C^1_{\mathrm{\rm loc}}(\mathbb{R}^n)$ to some function $\tilde{u}_\infty$. By construction we have that $\tilde{u}_\infty\in\mathcal D^{1,p}(\mathbb{R}^n)$, $\tilde{u}_\infty(\mathcal O)=1$, and $\tilde{u}_\infty$ is a weak solution of 
	\begin{equation}\label{3.13_vera}
	-\Delta^H_p \tilde{u}_\infty=  \, \tilde{u}_\infty^{p^\ast-1} \quad \text{ in } \, \mathbb{R}^n\, .
	\end{equation}

	- Case $(ii)$: consider a ball $B_{R}$. Then for every compact set $K\subset\subset B_{R}\cap \mathcal C$  there exists $\overline{\alpha}\in\mathbb{N}$ such that $K\subset\Sigma_\alpha\cap B_R$ for every $\alpha\geq\overline{\alpha}$. As in Case $(i)$, for every $\alpha\geq\overline{\alpha}$ the function $\tilde{u}_\alpha$ is a solution of \eqref{eq_tilde_u_vera} in $K$, and there exist a constant $C>0$ and a real number $\theta\in(0,1)$ such that
	\begin{equation}\label{3.12_vera_bis}
	||\tilde{u}_\alpha||_{C^{1,\theta}(K')}\leq C
	\end{equation}
	for any $\alpha\geq\overline{\alpha}$ and $K'\subset\subset K$. In addition, it follows by Remark \ref{rmk:Holder} that the functions $\tilde u_\alpha$ are uniformly $C^{0,\theta}$ inside $B_{R}\cap \overline{\mathcal C}$ for any $R>0$. 
	Hence, again Ascoli-Arzel\`a Theorem
	and a diagonal argument imply that $\{\tilde{u}_\alpha\}_{\alpha\in\mathbb{N}}$ converges (up to subsequence) in $C^0(B_{R}\cap \overline{\mathcal C})\cap C_{\rm loc}^1(B_{R}\cap \mathcal C)$ to some function $\tilde{u}_\infty$, for any $R>0$.  Taking the limit in the weak formulation of the equation, we obtain that $\tilde{u}_\infty\in \mathcal D^{1,p}(\mathcal C)$, $\tilde{u}_\infty(\mathcal O)=1$, and  $\tilde{u}_\infty$ is a weak solution of 
	\begin{equation}\label{3.13_vera_bis}
	\begin{cases}
	-\Delta_p^H \tilde{u}_\infty=   \, \tilde{u}_\infty^{p^\ast-1} & \text{ in } \mathcal C \\
	a(\nabla \tilde{u}_\infty) \cdot \nu =0 & \text{ on } \partial\mathcal C  \,. 
	\end{cases}
	\end{equation}

	We now notice that,
in both cases, for any $\rho>0$ we have 
	\begin{equation}\label{3.14_vera}
	||\tilde{u}_\alpha||_{L^{p^\ast}(\Sigma_\alpha\cap B_{\rho})}=||u||_{L^{p^\ast}(\Sigma\cap B_{\rho m_\alpha}(y_\alpha))}\, .
	\end{equation}
Also, by \eqref{3.6_anis}, since $r_k(u)<r''$ we get 
	\begin{equation}\label{3.15_vera}
	B_{\rho m_\alpha}(y_\alpha)\cap B_{r_k(u)}=\emptyset
	\end{equation}
	for $\alpha$ large. Thus, from \eqref{3.14_vera}, \eqref{3.15_vera}, and by definition of $r_k(u)$, we obtain 
	\begin{equation}\label{3.16_vera}
	||\tilde{u}_\alpha||_{L^{p^\ast}(\Sigma_\alpha\cap B_{\rho})}\leq k
	\end{equation}
	for $\alpha$ large. Thus, taking the limit in \eqref{3.16_vera} as $\alpha\rightarrow\infty$ and then as $\rho\rightarrow\infty$, yields
	\begin{equation}\label{3.17_vera}
	||\tilde{u}_\infty||_{L^{p^\ast}(\mathbb{R}^n)}\leq k \quad \text{or} \quad ||\tilde{u}_\infty||_{L^{p^\ast}(\mathcal C)}\leq k\,,
	\end{equation}
	in Case $(i)$ or Case $(ii)$, respectively. Since $k< k_0$ with $k_0>0$ as in Lemma \ref{lem:lower bound}, it follows by \eqref{3.13_vera} (resp. \eqref{3.13_vera_bis}) and \eqref{3.17_vera} that $\tilde{u}_\infty\equiv 0$ in Case $(i)$ (resp. Case $(ii)$), a contradiction to the fact that $\tilde{u}_\infty(\mathcal O)=1$. This completes the proof of the assertion of Step 1.

	\medskip

\noindent	$\bullet$
	\emph{Step 2: Let $u$ be a solution of \eqref{p_laplace_cono_i}. Then $u \in L^{\hat p-1,\infty}(\Sigma)$ for $\hat p := \frac{p(n-1)}{n-p}$.
	}
	
	Recall that, given a set $\Omega$ and $r\geq 1$, one defines 
the space $L^{r,\infty}(\Omega)$ as the set of all measurable functions $v:\Omega\rightarrow\mathbb{R}$ such that
		\begin{equation} \label{step2proof31}
		||v||_{L^{r,\infty}(\Omega)}:=\sup_{h>0}\left\{h\, \meas(\{|u|>h\})^{1/r}\right\} < \infty \,.
		\end{equation}
Using the Sobolev inequality in cones,  the proof of this step can be easily adapted from the case of $\R^n$ (see \cite[Lemma 2.2]{Vetois}) and for this reason is omitted.

	\medskip
	
\noindent	$\bullet$
	\emph{Step 3: Proof of \eqref{1.3-1.4_anis}}.
	
	The proof of this step closely follows the proof of \cite[Theorem 1.1]{Vetois}, which in turn uses \cite[Theorem 1.3]{Trudinger} and \cite[Theorem 5]{Serrin_local}. Even if \cite[Theorem 1.3]{Trudinger} and \cite[Theorem 5]{Serrin_local} are stated in a local setting, thanks to the homogeous Neumann boundary condition they can be easily extended to our setting. For this reason we only give a sketch of the proof, following the argument of \cite[Theorem 1.1]{Vetois}.
	
	Let $k$ and $r$ be as in Step 1. For any $R>0$ and $y\in\Sigma$, we define 
	\begin{equation}\label{4.1_anis}
	u_R(y):=R^{\frac{n-p}{p-1}}u(Ry)\, .
	\end{equation}
	From \eqref{p_laplace_cono} we obtain 
	\begin{equation}\label{4.4_anis}
	-\Delta^H_p u_{R}=R^{-\frac{p}{p-1}} u_{R}^{p^\ast-1} \qquad \text{in} \, \Sigma\, .
	\end{equation}
	Also, writing $u_R^{p^\ast-1}=u_R^{p^\ast-p}u_R^{p-1}$ and using \eqref{lorenzino}, we have
	\begin{equation}\label{4.5_anis}
	R^{-\frac{p}{p-1}}u_R^{p^\ast-1}\leq K_0^{p^\ast-p}u_R^{p-1} \quad \text{ in} \, \overline{\Sigma}\setminus B_{1} \,,
	\end{equation}
	provided that $R\geq r$. Thus, it follows from \eqref{4.4_anis}, \eqref{4.5_anis}, and  \cite[Theorem 1.3]{Trudinger}, that for any $\varepsilon>0$ it holds
	\begin{equation}\label{4.6_anis}
	||u_R||_{L^\infty(\Sigma\cap(B_{4}\setminus B_{2}))}\leq C_\varepsilon||u_R||_{L^{p-1+\varepsilon}(\Sigma\cap(B_{5}\setminus B_{1}))}
	\end{equation}
	for some constant $C_\varepsilon>0$. We fix $\varepsilon_0=\varepsilon_0(n,p)$ such that $0<\varepsilon_0<\hat p-p$, where $\hat p$ is as in Step 2. Since
	\begin{equation*}
	||u_R||_{L^{p-1+\varepsilon_0}(\Sigma\cap(B_{5}\setminus B_{1}))}\leq C_0 ||u_R||_{L^{\hat p-1,\infty}(\Sigma\cap(B_{5}\setminus B_{1}))} \, ,
	\end{equation*}
	for $C_0=C_0(n,p)$, recalling Step 2 we obtain that 
	\begin{equation}\label{4.8_anis}
	||u_R||_{L^\infty(\Sigma\cap(B_{4}\setminus B_{2}))}\leq C_1 
	\end{equation}
	for some constant $C_1$. Hence, by \eqref{4.4_anis}, \eqref{4.8_anis}, and elliptic regularity theory for $p$-Laplacian type equations \cite{DiBenedetto,Tolksdorf}, we get 
	\begin{equation}\label{4.9_anis}
	||\nabla u_R||_{L^\infty(\Sigma \cap (B_{7/2}\setminus B_{5/2}))}\leq C_2
	\end{equation}
	for some constant $C_2$. Here we notice that, even if \eqref{4.9_anis} is proved in \cite[Section 3]{DiBenedetto} in a local setting (see also \cite{CM_Global}, where the authors prove global Lipschitz regularity in convex domains for the case when $H$ coincides with the Euclidean norm), the argument easily extends to our setting by an approximation argument.
	Indeed, as in the proof of Proposition \ref{Lemma_Caccioppoli} below, one can work in regularized domains and, because of the presence of the boundary, with respect to \cite[Section 3]{DiBenedetto} it appears an extra boundary term. However, this can be dropped since the second fundamental form of $\partial \Sigma$ is nonnegative definite (compare with \eqref{int2}-\eqref{int4} below, or with \cite[Proof of Theorem 1.2, Step 1]{CM_Global}).

	Finally, for any $x\in\mathbb{R}^n\setminus B_{3r}$, applying \eqref{4.8_anis} and \eqref{4.9_anis} with $R=|x|/3$ we obtain
	\begin{equation}\label{4.10_anis}
	u(x)\leq C_3|x|^{\frac{p-n}{p-1}} \quad \text{ and } \quad |\nabla u(x)|\leq C_3|x|^{\frac{1-n}{p-1}}
	\end{equation}
	for some constant $C_3$. Since $u$ and $\nabla u$ are uniformly bounded in $B_{3r}$,  \eqref{1.3-1.4_anis} follows. Finally, to prove the lower bound in \eqref{1.3-1.4_anis} one argues as in \cite[pages 159-160]{Vetois}.
%
%
%
%
\end{proof}

\subsection{Asymptotic estimates on higher order derivatives} \label{subsect_caccioppoli}
By using a Caccioppoli-type inequality, in this subsection we prove Proposition \ref{Lemma_Caccioppoli} below which will be useful in the proof of Theorem \ref{thm_main}. In particular it will avoid the use of an asymptotic lower bound on $|\nabla u|$, which is crucial in \cite{Sc}.


\begin{proposition}\label{Lemma_Caccioppoli}
Let $\Sigma$ be a convex cone, and let $u$ be a solution to \eqref{p_laplace_cono} with $a(\cdot)$ given by \eqref{a_def}, where $H$ satisfies the assumptions of Theorem \ref{thm_main}. Then $a(\nabla u)\in W^{1,2}_{\rm loc}(\overline{\Sigma})$, and for any $\gamma\in\mathbb{R}$ the following asymptotic estimate holds:
%
%
\begin{equation}\label{dis_caccioppoli_r}
\int_{B_{r}\cap\Sigma}|\nabla(a(\nabla u))|^2u^{\gamma}\, dx\leq  C\Big(1+r^{-n-\gamma\frac{n-p}{p-1}} \Big)\qquad \forall\,r \geq 1,
\end{equation}
where $C$ is a positive constant independent of $r$.
\end{proposition}

\begin{proof}
The estimate \eqref{dis_caccioppoli_r} is obtained by using a Caccioppoli-type inequality. We argue by approximation, following the approach in \cite{AKM,CM_top}. 

We approximate $\Sigma$ by a sequence of convex cones $\{\Sigma_k\}$ such that $\Sigma_k \subseteq \Sigma$ and $\partial \Sigma_k \setminus \{\mathcal O\}$ is smooth. 
Also, we fix a point $\bar x \in \cap_k\Sigma_k$, and for $k$ fixed we let $u_k$ be the solution of\footnote{The function $u_k$ can be found by considering first the minimizer $v_{k,R}$ of the minimization problem
$$
\min_v\left\{\int_{\Sigma_k\cap B_R}\left[\frac1p H(\nabla v)^p -u^{p^*-1}v\right]\,dx \,:\,\text{$v=0$ on $\Sigma_k\cap \partial B_R$}\right\},
$$
then setting $u_{k,R}(x):=v_{k,R}(x)+u(\bar x)-v_{k,R}(\bar x)$, and finally taking the limit of $u_{k,R}$ as $R\to \infty$ (note that  the functions $\tilde u_{k,R}$ are uniformly $C^{1,\theta}$ in every compact subset of $\Sigma$, and uniformly H\"older continuous up to the boundary).
 }
\begin{equation}\label{approx_1bis}
\begin{cases}
\diver(a(\nabla u_k)) + u^{p^*-1} = 0 & \text{ in } \Sigma_k
\\
u_k(\bar x)=u(\bar x)\\
a(\nabla  u_k)\cdot\nu=0 & \text{ on } \partial\Sigma_k\,.
\end{cases}
\end{equation}
Set
\begin{equation} \label{ajz}
a^\ell(z):=(a\ast\phi_\ell)(z)\qquad \text{for $z\in\mathbb{R}^n$} \, ,
\end{equation}
where $\{\phi_\ell\}$ is a family of radially symmetric smooth mollifiers. Standard properties of convolution and the fact $a(\cdot)$ is continuous imply
$a^\ell \rightarrow a $ uniformly on compact subset of $\mathbb{R}^n$. From \cite[Lemma 2.4]{fusco} we have that $a^\ell$ satisfies the first condition in \eqref{ellipticity_approx} with $s$ replaced by $s_\ell$, where $s_\ell\to 0$ as $\ell\to \infty$.
In addition, since 
$$
\dfrac{1}{\tilde{\alpha}}(|z|^2+s_\ell^2)^{\frac{p-2}{2}}|\xi|^2\leq\nabla a^\ell(z)\xi\cdot\xi\, , \quad \text{ for every $\xi, z\in\mathbb{R}^n$,}
$$
for some $\tilde{\alpha}>0$, we obtain that $a^\ell$ satisfies also the second condition in \eqref{ellipticity_approx}.

Let $u_{k,\ell}$ be a solution of 
\begin{equation}\label{approx_1}
\begin{cases}
\diver(a^\ell(\nabla u_{k,\ell})) + u^{p^*-1} = 0 & \text{ in } \Sigma_k
\\
a^\ell(\nabla u_{k,\ell})\cdot\nu=0 & \text{ on } \partial\Sigma_k 
\end{cases}
\end{equation}
(this solution can be constructed analogously to $u_k$).

%
%

We notice that $u_{k,\ell}$ is unique up to an additive constant.
Also, because $u$ is locally bounded, the functions $u_{k,\ell}$ are $C^{1,\theta}_{\rm loc}(\overline\Sigma_k \setminus \{\mathcal O\}) \cap C^{0,\theta}_{\rm loc}(\overline \Sigma_k)$, uniformly in $\ell$. In particular, assuming without loss of generality that $u_{k,\ell}(\bar x)=u(\bar x)$ for some fixed point $\bar x \in \Sigma_k$, as $\ell \to \infty$ one sees that $u_{k,\ell}$ converges in $C^1_{\rm loc}$ to the unique solution $\bar u_k$ of 
\begin{equation}\label{approx_1bis}
\begin{cases}
\diver(a(\nabla \bar u_k)) + u^{p^*-1} = 0 & \text{ in } \Sigma_k
\\
\bar u_k(\bar x)=u(\bar x)\\
a(\nabla \bar u_k)\cdot\nu=0 & \text{ on } \partial\Sigma_k \,.
\end{cases}
\end{equation}
Since $u_k$ is also a solution of the problem above, it follows by uniqueness that $\bar u_k=u_k$ and therefore
$u_{k,\ell}$ converges to $u_k$ as $\ell \to \infty$. Analogously, $u_k\to u$ as $k\to \infty$.

Given $R>1$ large, we define 
$$
\Omega_k:=\Sigma_k\cap B_{R}\, , \quad \Gamma_{k,0}:=\Sigma_k\cap \partial B_{R} \, , \quad \Gamma_{k,1}:=\partial\Sigma_k\cap B_{R}\, .
$$
Note that, since $u$ is uniformly positive inside $\Sigma$ (see Proposition \ref{Vetosi_anis}), for $k$ large enough (depending on $R$) also $u_k$ is uniformly positive inside $\Omega_k$,
and hence for $\ell$ large enough we have that $u_{k,\ell}$ is also uniformly positive inside $\Omega_k$. In the sequel we shall always assume that $k$ and $\ell$ are sufficiently large so that this positivity property holds.
We now fix $k$ and deal with the functions $u_{k,\ell}$. To simplify the notation, we shall drop the dependency on $k$ and we write $u_\ell, \Sigma,\Omega,\Gamma_0,\Gamma_1$ instead of $u_{k,\ell}, \Sigma_{k},\Omega_{k},\Gamma_{k,0},\Gamma_{k,1}$, respectively.

	The  idea is to prove a Caccioppoli-type inequality for $u_\ell$ and then let $\ell\to \infty$. Since $u_\ell$ solves a non-degenerate equation, we have that $u_\ell\in C^1\cap W^{2,2}_{\rm loc}(\overline\Sigma)$ and furthermore we have $a^\ell(\nabla u_\ell)\in W^{1,2}_{\rm loc}(\overline\Sigma)$. In addition, since $\Sigma$ is smooth outside the origin, $u_\ell$ is of class $C^2$ in $\overline\Omega$ away from $\Gamma_1\cup\{\mathcal O\}$. 

Multiply \eqref{approx_1} by $\psi\in C^\infty_c(B_{R}\setminus B_{1/R})$ and integrate over $\Omega$ to get
\begin{equation*}
\int_{\Omega}\diver(a^\ell(\nabla u_\ell))\psi\, dx =-\int_{\Omega} u^{p^*-1}\psi\, dx,
\end{equation*}
that together with the divergence theorem gives
\begin{equation}\label{previous}
\begin{aligned}
-\int_{\Omega} a^\ell(\nabla u_\ell)\cdot\nabla\psi\, dx+\int_{\partial\Omega}\psi a^\ell(\nabla u_\ell)\cdot\nu\, d\sigma=-\int_{\Omega} u^{p^*-1}\psi\, dx\, .
\end{aligned}
\end{equation}
Since
$$
\int_{\partial\Omega}\psi a^\ell(\nabla u_\ell)\cdot\nu\, d\sigma=\int_{\Gamma_{1}}\psi a^\ell(\nabla u_\ell)\cdot\nu\, d\sigma + \int_{\Gamma_{0}}\psi a^\ell(\nabla u_\ell)\cdot\nu\, d\sigma\, ,
$$
from the fact that $\psi\in C^{\infty}_c(B_{R}\setminus B_{1/R})$ and from the boundary condition in \eqref{approx_1}, we obtain that  the second term in \eqref{previous} vanishes; hence \eqref{previous} becomes
\begin{equation}\label{previous_bis}
\begin{aligned}
-\int_{\Omega} a^\ell(\nabla u_\ell)\cdot\nabla\psi\, dx=-\int_{\Omega} u^{p^*-1}\psi\, dx\, .
\end{aligned}
\end{equation}
Let  $\varphi \in C^\infty_c(B_{R}\setminus B_{1/R})$, and for
 $\delta>0$ small define the set  
$$
\Omega_{\delta}:=\{ x\in\Omega \, : \,  \dist(x,\partial\Omega)>\delta\}\, .
$$ 
Since $\Omega\cap {\rm supp}(\varphi)$ is smooth, for $\delta$ small enough we see that 
$\Omega_{\delta}\setminus\Omega_{2\delta}$ is of class $C^{\infty}$ inside the support of $\varphi$.
In particular, every point
$x \in (\Omega_{\delta}\setminus\Omega_{2\delta})\cap {\rm supp}(\varphi)$ can be written as 
$$
x=y - |x-y|\nu(y)
$$
where $y=y(x)\in \partial \Omega_{\delta}$ is the projection of $x$ on $\partial \Omega_{\delta}$ and $\nu(y)$ is the outward normal to $\partial \Omega_{\delta}$ at $y$. 
Moreover the set $(\Omega_{\delta}\setminus\Omega_{2\delta})\cap {\rm supp}(\varphi)$ can be parametrized on $\partial\Omega_{\delta}$ by a $C^1$ function $g$ (see \cite[Formula 14.98]{GT}).

Let $\zeta_\delta:\Omega \to [0,1]$ be a cut-off function such that $\zeta_\delta=1$ in $\Omega_{2\delta}$, $\zeta_\delta=0$ in $\Omega\setminus\Omega_{\delta}$, and 
$$
\nabla\zeta_\delta(x)=-\frac{1}{\delta}\nu(y(x))\qquad \text{inside }\Omega_{\delta}\setminus\Omega_{2\delta}\,.
$$
Using $\psi=\partial_m(\varphi\zeta_\delta)$ in \eqref{previous_bis} with $m\in \{1,\ldots,n\}$ and integrating by parts, we get 
\begin{equation*}
\begin{aligned}
\sum_{i=1}^n\left(\int_{\Omega} \partial_m a^\ell_i(\nabla u_\ell)\zeta_\delta\partial_i\varphi\, dx+\int_{\Omega} \partial_m a^\ell_i(\nabla u_\ell)\varphi\partial_i\zeta_\delta\, dx\right)=\int_{\Omega} \partial_m(u^{p^*-1})\varphi\zeta_\delta\, dx\, ,
\end{aligned}
\end{equation*}
where we use the notation $a^\ell=(a^\ell_1,\ldots,a^\ell_n)$ to denote the components of the vector field $a^\ell$.

Observe that, from the definition of $\zeta_\delta$, we have  
\begin{equation*}\label{lim1}
\lim_{\delta\rightarrow 0}\int_{\Omega} \partial_m a^\ell_i(\nabla u_\ell)\zeta_\delta\partial_i\varphi\, dx=\int_{\Omega} \partial_m a^\ell_i(\nabla u_\ell)\partial_i\varphi\, dx\, .
\end{equation*}
Also, if we set
$$
f(x)=\partial_m a^\ell_i(\nabla u_\ell(x))\varphi(x)\,,
$$
by the coarea formula we have
\begin{equation*}
\begin{aligned}
\int_{\Omega_\delta\setminus\Omega_{2\delta}} f\partial_i\zeta_\delta\, dx &  = - \frac{1}{\delta} \int_{\Omega_{\delta}\setminus\Omega_{2\delta}}\nu_{i}(y(x))f dx \\
& =  - \frac{1}{\delta} \int_{\delta}^{2 \delta} dt \int_{\partial\Omega_{\delta}} \nu_{i}(y(x))f(y-t\nu(y))|{\rm det}(Dg)|d\sigma(y) \\ & =  -  \int_{1}^{2} ds \int_{\partial\Omega_{s\delta}}  f(y - s \delta \nu(y))  \nu_i(y) |{\rm det}(Dg)|d\sigma(y) \, .
\end{aligned}
\end{equation*}
 Since $f \in C^0$, we can pass to the limit and obtain
\begin{equation*}
\lim_{\delta \to 0} \int_{\Omega} \partial_m a^\ell_i(\nabla u_\ell)\varphi\partial_i\zeta_\delta\, dx = - \int_{\partial\Omega}  \partial_m a^\ell_i(\nabla u_\ell)\varphi\nu_i d \sigma \,.
\end{equation*}
Hence, we proved that
\begin{equation}\label{11.16}
\begin{aligned}
\sum_{i=1}^n\left(\int_{\Omega} \partial_m a^\ell_i(\nabla u_\ell)\partial_i\varphi\, dx- \int_{\partial\Omega}  \partial_m a^\ell_i(\nabla u_\ell)\varphi\nu_i d \sigma\right)=\int_{\Omega} \partial_m(u^{p^*-1})\varphi\, dx\, .
\end{aligned}
\end{equation}
Now, 
let 
$$
\Omega_{\delta}^t:=\{ x\in\Omega_{\delta} \, : \, \dist(x,\partial\Omega_{\delta})>t\}\, .
$$
We notice that, if $x \in (\Omega_{\delta}\setminus\Omega_{2\delta})\cap {\rm supp}(\varphi)$ with $x=y - t \nu(y)$, then $x\in\partial\Omega_{\delta}^t$ and the outward normal to $\partial\Omega_{\delta}^t$ at $x$ coincides with the outward normal to $\partial\Omega_{\delta}$ at $y$. Hence, by writing $\nu(x)$ in place of $\nu(y)$, we have
\begin{equation}\label{int2}
\begin{aligned}
\partial_m a^\ell_i(\nabla u_\ell(x))\varphi(x)\nu_i(x)=&\,\varphi(x)\partial_m(a^\ell(\nabla u_\ell(x))\cdot\nu(x))\\ 
&-\varphi(x)a^\ell_i(\nabla u_\ell(x))\partial_m\nu_i(x)\, .
\end{aligned}
\end{equation}
Now, we take a cut-off function $\eta\in C^{\infty}_c(B_{R}\setminus B_{1/R})$, and for $m\in\{1,\dots,n\}$ we set $\varphi=a^\ell_m(\nabla u_\ell)u_\ell^{\gamma}\eta^2$ where $\gamma\in\mathbb{R}$, and in \eqref{int2} we obtain
\begin{equation}\label{int3}
\begin{aligned}
\partial_m a^\ell_i(\nabla u_\ell(x))\varphi(x)\nu_i(x)=&\,a^\ell_m(\nabla u_\ell(x))u_\ell^{\gamma}(x)\eta^2(x)\partial_m\bigl(a^\ell(\nabla u_\ell(x))\cdot\nu(x)\bigr)\\ 
&-a^\ell_m(\nabla u_\ell(x))u_\ell^{\gamma}(x)\eta^2(x)a^\ell_i(\nabla u_\ell(x))\partial_m\nu_i(x) \, .
\end{aligned}
\end{equation}
We notice that $\partial_m\nu_i (x)$ is the second fundamental form $\mathrm{II}_{x}^{t}$ of $\partial\Omega_{\delta}^t$ at $x$:
$$
\sum_{i,m=1}^n\partial_m\nu_i (x)a^\ell_i(\nabla u_\ell(x))a^\ell_m(\nabla u_\ell(x))=\mathrm{II}_{x}^{t}(a^\ell(\nabla u_\ell(x)),a^\ell(\nabla u_\ell(x)))\, .
$$
Since the cone $\Sigma$ is convex then $\mathrm{II}_{x}^{t}$ is non-negative definite, which implies that
\begin{equation}\label{convexity1}
\sum_{i,m=1}^n\partial_m\nu_i (x)a^\ell_i(\nabla u_\ell(x))a^\ell_m(\nabla u_\ell(x))\geq 0\, .
\end{equation}
Hence \eqref{int3} becomes
\begin{equation}\label{int4}
\begin{aligned}
\sum_{i,m=1}^n\partial_m a^\ell_i(\nabla u_\ell(x))\varphi(x)\nu_i(x)
\leq \sum_{i,m=1}^na^\ell_m(\nabla u_{n}(x))u_\ell^{\gamma}(x)\eta^2(x)\partial_m\bigl(a^\ell(\nabla u_\ell(x))\cdot\nu(x)\bigr)\, ,
\end{aligned}
\end{equation}
and so, with the choice $\varphi=a^\ell_m(\nabla u_\ell)u_\ell^{\gamma}\eta^2,$ we obtain
\begin{equation*}
\begin{aligned}
 \sum_{i,m=1}^n\int_{\partial\Omega}  \partial_m a^\ell_i(\nabla u_\ell)\varphi\nu_i d \sigma &\leq \sum_{i,m=1}^n \int_{\partial\Omega}u_\ell^{\gamma}\eta^2a^\ell_m(\nabla u_\ell)\partial_m\bigl(a^\ell(\nabla u_\ell)\cdot\nu\bigr)\, dx\\
 &=\sum_{i=1}^n \int_{\partial\Omega}u_\ell^{\gamma}\eta^2\,a^\ell(\nabla u_\ell)\cdot \nabla \bigl(a^\ell(\nabla u_\ell)\cdot\nu\bigr)\, dx =0\, ,
 \end{aligned}
\end{equation*}
where the last equality follows from the condition $a^\ell(\nabla u_\ell)\cdot\nu=0$ on $\partial\Sigma$. Indeed, this condition  implies that $a^\ell(\nabla u_\ell)$ is a tangent vector-field and that the tangential derivative of $a^\ell(\nabla u_\ell)\cdot\nu$ vanishes on $\partial\Sigma$.

Hence, recalling \eqref{11.16}, we proved that
\begin{equation}\label{Jannacci_bis}
\begin{aligned}
\sum_{i,m=1}^n\int_{\Omega} \partial_m a^\ell_i(\nabla u_\ell)\partial_i\left(a^\ell_m(\nabla  u_\ell)u_\ell^{\gamma}\eta^2\right)\, dx\leq n\int_{\Omega}  |\nabla (u^{{p^*-1}})||a^\ell(\nabla  u_\ell)|u_\ell^{\gamma}\eta^2\, dx \,.
\end{aligned}
\end{equation}
Inequality \eqref{Jannacci_bis} can be used in place of Equation (4.11) in \cite[Proof of Theorem 4.1]{AKM}, and by arguing as in \cite{AKM} we obtain
\begin{multline*}
\int_{\Omega}|\nabla(a^\ell(\nabla u_\ell))|^2\eta^2 u_\ell^{\gamma}\, dx\leq \\ C\int_{\Omega}|\nabla(a^\ell(\nabla u_\ell))||a^\ell(\nabla u_\ell)|\eta u_\ell^{\frac{\gamma}{2}}|\nabla(\eta u_\ell^{\frac{\gamma}{2}})|\, dx +C \int_{\Omega} |\nabla (u^{{p^*-1}})||a^\ell(\nabla  u_\ell)|u_\ell^{\gamma}\eta^2\, dx \,.
\end{multline*}
From H\"older and Young inequalities, for any $\epsilon \in (0,1)$  we can bound
\begin{multline*}
C\int_{\Omega}|\nabla(a^\ell(\nabla u_\ell))||a^\ell(\nabla u_\ell)|\eta u_\ell^{\frac{\gamma}{2}}|\nabla(\eta u_\ell^{\frac{\gamma}{2}})|\, dx
\\
\leq C\epsilon \int_{\Omega}|\nabla(a^\ell(\nabla u_\ell))|^2\eta^2 u_\ell^{\gamma}\, dx
+\frac{C}{\epsilon}\int_{\Omega}|a^\ell(\nabla u_\ell)|^2|\nabla(\eta u_\ell^{\frac{\gamma}{2}})|^2\, dx,
\end{multline*}
so choosing $\epsilon$ small enough such that $C\epsilon =1/2$, we obtain 
$$
\int_{\Omega}|\nabla(a^\ell(\nabla u_\ell))|^2\eta^2u_\ell^{\gamma}\, dx\leq C\int_{\Omega}|a^\ell(\nabla u_\ell)|^2|\nabla(\eta u_\ell^{\frac{\gamma}{2}})|^2\, dx+C \int_{\Omega} |\nabla (u^{{p^*-1}})||a^\ell(\nabla  u_\ell)|u_\ell^{\gamma}\eta^2\, dx.
$$
Recall that here $\eta\in C^{\infty}_c(B_{R}\setminus B_{1/R})$. However, by approximation the same property holds for any $\eta \in C^\infty_c(\R^n)$.

Now, we recall that we were writing $u_\ell$ in place of $u_{k,\ell}$.
Then, since $u_{k,\ell} \to u_k$ in $C^1_{\rm loc}$ and $a^\ell\to a$ locally uniformly, we can let $\ell\to \infty$ to deduce that
\begin{equation} \label{robertocarlos}
\int_{\Omega_k}|\nabla(a(\nabla u_k))|^2\eta^2u_k^{\gamma}\, dx\leq C\int_{\Omega_k}|a(\nabla u_k)|^2|\nabla(\eta u_k^{\frac{\gamma}{2}})|^2\, dx+C \int_{\Omega_k} |\nabla (u^{{p^*-1}})||a(\nabla  u_k)|u_k^{\gamma}\eta^2\, dx.
\end{equation}
In particular, taking $\gamma=0$, \eqref{robertocarlos} proves that $a(\nabla u_k)\in W^{1,2}_{\rm loc}(\overline{\Sigma}_k)$, and $\{a(\nabla u_k)\}_{k \in \N}$ is uniformly bounded in $W^{1,2}_{\rm loc}$. Hence, letting $k \to\infty$ in \eqref{robertocarlos} we obtain 
\begin{equation*} 
\int_{\Omega}|\nabla(a(\nabla u))|^2\eta^2u^{\gamma}\, dx\leq C\int_{\Omega}|a(\nabla u)|^2|\nabla(\eta u^{\frac{\gamma}{2}})|^2\, dx+C \int_{\Omega} |\nabla (u^{{p^*-1}})||a(\nabla  u)|u^{\gamma}\eta^2\, dx.
\end{equation*}
Finally, the asymptotic estimate \eqref{dis_caccioppoli_r} follows from \eqref{1.3-1.4_anis}. 
\end{proof}

\section{Proof of Theorem \ref{thm_main}} \label{sect_proof}
As already mentioned in the introduction, we consider the auxiliary function 
\begin{equation}\label{def_v}
v= u^{-\frac{p}{n-p}}
\end{equation}
where $u$ is a solution of \eqref{p_laplace_cono}. A straightforward computation shows that $v>0$ satisfies the following problem
\begin{equation}\label{eq_v_bis_anis}
\begin{cases}
\Delta^H_p v = f(v,\nabla v) & \text{in } \Sigma \\
a(\nabla v) \cdot \nu = 0 & \text{on } \partial \Sigma  \,,
\end{cases}
\end{equation}
where $\Delta^H_p v=\diver(a(\nabla v))$ with $a(\xi)$ as \eqref{a_def},
and we set
\begin{equation} \label{f_v_nablav}
f(v,\nabla v)  = \left(\frac{p}{n-p} \right)^{p-1}\dfrac{1}{v}+\frac{n(p-1)}{p}\dfrac{H^p(\nabla v)}{v} \,.
\end{equation}
It is clear that $v$ inherits some properties from $u$. In particular $v \in C^{1,\theta}_{\rm loc}$, and it follows from Proposition \ref{Vetosi_anis} that 
there exist constants $C_0,C_1>0$  such that 
\begin{equation} \label{v_asympt}
C_0  |x|^{-\frac{p}{p-1}}  \leq v(x) \leq C_1 |x|^{-\frac{p}{p-1}} 
\end{equation}
and 
\begin{equation} \label{nabla_v_asympt}
|\nabla v(x)| \leq C_1   |x|^{-\frac{1}{p-1}} 
\end{equation}
for $|x|$ sufficiently large. Higher regularity results for $v$ are summarized in the following lemma.

\begin{lemma}\label{lemma_caccioppoli}
Let $v$ be given by \eqref{def_v}. Then, for every $\sigma\in\mathbb{R}$, the asymptotic estimate
\begin{equation} \label{caccioppoli_asymptotic}
\int_{B_r\cap\Sigma}|\nabla(a(\nabla v))|^2v^\sigma\, dx\leq C\Big(1+r^{n+\frac{\sigma p}{p-1}}\Big)\qquad \forall\,r\geq 1
\end{equation}
holds.
\end{lemma}

\begin{proof}
We notice that 
\begin{equation*}
a(\nabla v)= -\left(\dfrac{p}{n-p}\right)^{p-1}u^{-\frac{n(p-1)}{n-p}}a(\nabla u)
\end{equation*}
and  
\begin{equation*}
\nabla(a(\nabla v)) =-\left(\dfrac{p}{n-p}\right)^{p-1}\left[u^{-\frac{n(p-1)}{n-p}}\nabla(a(\nabla u)) - \dfrac{n(p-1)}{n-p}u^{\frac{p(1-n)}{n-p}}\nabla u\, \otimes  a (\nabla u)\right] \,,
\end{equation*}
so it follows from Proposition \ref{Lemma_Caccioppoli} that
\begin{equation} \label{a_nablav_W12}
a(\nabla v) \in W^{1,2}_{\rm loc}(\overline \Sigma) \,.
\end{equation}
Finally, the asymptotic estimate  \eqref{caccioppoli_asymptotic} follows from \eqref{dis_caccioppoli_r} and \eqref{1.3-1.4_anis}.
\end{proof}

\subsection{An integral inequality}
In this subsection, by using the convexity of the cone, we show that $v$ satisfies an integral inequality.

We recall that the second symmetric function $S^2(M)$ of a $n\times n$ matrix $M=(m_{ij})$  is the sum of all the principal minors of $A$ of order two, and we have 
\begin{equation}\label{eq200}
S^2(M)=\frac12\sum_{i,j}S^2_{ij}(M)m_{ij}\,,
\end{equation}
where 
$$
S^2_{ij}(M)=-m_{ji} + \delta_{ij} \tr (M) \,.
$$
As proved in \cite[Lemma 3.2]{CS}, given two symmetric matrices $B,C\in \mathbb{R}^{n\times n}$ with $B$ positive semidefinite, and by setting $M=BC$, we have the following Newton's type inequality:
\begin{equation}\label{newtonIneq}
S^2{(M)}\le\frac{n-1}{2n}\tr(M)^2\, .
\end{equation}
Moreover, if ${\tr} (M)\neq 0$ and equality holds in (\ref{newtonIneq}), then 
\begin{equation*}
M=\frac{{\tr}(M)}{n}\, {\rm Id} \,,
\end{equation*}
and $B$ is positive definite. As we will describe later, we will apply \eqref{newtonIneq} to the matrix $M=\nabla [a(\nabla v)]$.

We start from the following differential identity (see \cite{BiCi}).
We use the Einstein convention of summation over repeated indices.

\begin{lemma}\label{lemma_BC}
	Let $v$ be a positive function of class $C^3$ and let $V:\mathbb{R}^n\rightarrow\mathbb{R}^+$ be of class $C^3(\mathbb{R}^n)$ and such that $V(\nabla v)\diver(\nabla V(\nabla v))$ can be continuously extended to zero at $\nabla v=0$. 
	Let 
	\begin{equation}\label{def_W}
	W=\nabla[\nabla_\xi V(\nabla v)]=V_{\xi_i\xi_j}(\nabla v)v_{ij}\, .
	\end{equation}
	Then, for any $\gamma\in\mathbb{R}$ we have
		\begin{equation}
	2v^\gamma S^2(W)=\diver(v^\gamma S^2_{ij}(W)V_{\xi_i}(\nabla v)) - \gamma v^{\gamma-1}S^2_{ij}(W)V_{\xi_i}(\nabla v)v_j
	\end{equation}
	and 
	\begin{equation}\label{BC_generic}
	\begin{aligned}
	\diver&\bigl(v^\gamma S_{ij}^2(W)V_{\xi_i}(\nabla v)+\gamma(p-1)v^{\gamma-1}V(\nabla v) V_{\xi_j}(\nabla v)\bigr) \\ 
	= &\,2v^\gamma S^2(W) +\gamma(\gamma-1)(p-1)v^{\gamma-2}V(\nabla v)V_{\xi_i}(\nabla v)v_i \\
	&+\gamma v^{\gamma-1}\left((p-1)V(\nabla v)+V_{\xi_i}(\nabla v)v_i\right){\tr}(W) \\
	&+\gamma v^{\gamma-1}\left((p-1)V_{\xi_i}(\nabla v)V_{\xi_j}(\nabla v)v_{ij}+V_{\xi_j\xi_l}(\nabla v)v_{li}V_{\xi_i}(\nabla v)v_j\right)\, .
	\end{aligned}
	\end{equation}
	In particular, if $H$ is a norm and 
	\begin{equation}\label{plus}
	V(\xi)=\dfrac{H^p(\xi)}{p} \quad \text{ for $p>1$ and $\xi\in\mathbb{R}^n$}\, ,
	\end{equation}
	then 
	\begin{equation}\label{BC_anis}
	\begin{aligned}
	2v^\gamma S^2(W)=&\,\diver\bigl(v^\gamma S_{ij}^2(W)V_{\xi_i}(\nabla v)+\gamma(p-1)v^{\gamma-1}V(\nabla v)\nabla_\xi V(\nabla v)\bigr) \\
	&-\gamma(\gamma-1)p(p-1)v^{\gamma-2}V^2(\nabla v)-\gamma(2p-1)v^{\gamma-1}V(\nabla v)\Delta^H_p v\, ,
	\end{aligned}
	\end{equation}
	where 
$\Delta_p^H v=\diver( a(\nabla v))$ and $a(\cdot)$ is given by \eqref{a_def}.
	Observe that, in this particular case,
	$
	W(x): = \nabla[a(\nabla v(x))]$.
\end{lemma}

\begin{proof}
	See \cite[Lemma 4.1]{BiCi}.
\end{proof}

The idea is to apply the above lemma to the function $v$ solving \eqref{eq_v_bis_anis} and integrate the identity above on $\Sigma$. Due to the lack of regularity of $v$, Lemma \ref{lemma_BC} cannot be applied directly but we can still prove its integral counterpart.

\begin{lemma} \label{lemma_integr1}
Let $v$ be given by \eqref{def_v}, let $V$ be as in \eqref{plus}, and $W$ as in \eqref{def_W}.  Then, for any $\varphi\in C^{\infty}_c(\Sigma)$, we have 
\begin{equation}\label{BC_integral}
\begin{aligned}
\int_{\Sigma}&\left(2v^\gamma S^2(W)+\gamma(\gamma-1)p(p-1)v^{\gamma-2}V^2(\nabla v)+\gamma(2p-1)v^{\gamma-1}V(\nabla v)\Delta^H_p v\right) \varphi \\=&-\int_{\Sigma}\varphi_{j}\bigl(v^\gamma S_{ij}^2(W)V_{\xi_i}(\nabla v)+\gamma(p-1)v^{\gamma-1}V(\nabla v)V_{\xi_j}(\nabla v)\bigr) \, .
\end{aligned}
\end{equation}
\end{lemma}

\begin{proof} We argue by approximation. So, first we extend $v$ as $0$ outside $\Sigma$, and then for $\varepsilon>0$ we define $v^\varepsilon=v\ast\rho^\varepsilon$ and $V^\varepsilon=V\ast\rho^\varepsilon$, where $\rho^\varepsilon$ is a standard mollifier. Also, we set $a^\varepsilon=\nabla V^\varepsilon$ and $W^\varepsilon=(w^\varepsilon_{ij})_{i,j=1,\dots,n}$ where $w_{ij}^\varepsilon=\partial_j (a^\varepsilon_i(\nabla v^\varepsilon))$.

Since $V\in C^1(\mathbb{R}^n)$ then $a^\varepsilon_i=a_i\ast\rho^\varepsilon$ for $i=1,\dots,n$, where $a$ is given by \eqref{a_def}.
Also, since $a(\nabla v)\in W^{1,2}_{\rm loc}(\Sigma)$, then $a^\varepsilon_i(\nabla v^\varepsilon)\rightarrow a_i(\nabla v)$ and $w_{ij}^\varepsilon\rightarrow w_{ij}$ in $L^2_{\rm loc}(\Sigma)$. 
	
Moreover, since $H_0(\nabla H(\xi))=1$ for any $\xi\in\mathbb{R}^n\setminus \{0\}$ we have that $H_0(a(\xi))=H^{p-1}(\xi)$, which implies that $pV(\xi)=H^{\frac{p}{p-1}}_0(a(\xi))$. Since $H^{\frac{p}{p-1}}_0$ is locally Lipschitz and $a(\nabla v)\in W^{1,2}_{\rm loc}(\Sigma)$ then $V(\nabla v)\in  W^{1,2}_{\rm loc}(\Sigma)$ and we have that $\partial_{x_j}(V^\varepsilon(\nabla v^\varepsilon))\rightarrow\partial_{x_j}(V(\nabla v))$ in $L^2_{\rm loc}(\Sigma)$.
Now we write \eqref{BC_generic} for the approximating functions $v^\varepsilon$, $V^\varepsilon$ and $W^\varepsilon$, we multiply by $\varphi\in C^{\infty}_c(\Sigma)$ and integrate over $\Sigma$.
Since $\varphi$ has compact support inside $\Sigma,$ it follows from the divergence theorem that \begin{equation}\label{BC_integral_approx}
\begin{aligned}
\int_{\Sigma}&\left(2(v^\varepsilon)^\gamma S^2(W^\varepsilon)+\gamma(\gamma-1)(p-1)(v^\varepsilon)^{\gamma-2}V^\varepsilon(\nabla v^\varepsilon)V^{\varepsilon}_{\xi_i}(\nabla v^\varepsilon)v^\varepsilon_i\right) \varphi \\
+\int_{\Sigma}&\gamma (v^\varepsilon)^{\gamma-1}\left((p-1)V^\varepsilon(\nabla v^\varepsilon)+V^\varepsilon_{\xi_i}(\nabla v^\varepsilon)v^\varepsilon_i\right){\tr}(W^\varepsilon)\varphi \\
+\int_{\Sigma}& \gamma (v^\varepsilon)^{\gamma-1}\left((p-1)V^\varepsilon_{\xi_i}(\nabla v^\varepsilon)V^\varepsilon_{\xi_j}(\nabla v^\varepsilon)v^\varepsilon_{ij}+V^\varepsilon_{\xi_j\xi_l}(\nabla v^\varepsilon)v^\varepsilon_{li}V^\varepsilon_{\xi_i}(\nabla v^\varepsilon)v^\varepsilon_j\right)\varphi \\
=&-\int_{\Sigma}\varphi_{j}\bigl((v^\varepsilon)^\gamma S_{ij}^2(W^\varepsilon)V^\varepsilon_{\xi_i}(\nabla v^\varepsilon)+\gamma(p-1)(v^\varepsilon)^{\gamma-1}V^\varepsilon(\nabla v^\varepsilon)V^\varepsilon_{\xi_j}(\nabla v^\varepsilon)\bigr) \, .
\end{aligned}
\end{equation} 
Since $V^\varepsilon_{\xi_i}(\nabla v^\varepsilon)v^\varepsilon_{ij}=\partial_{x_j}(V^\varepsilon(\nabla v^\varepsilon))$, recalling  \eqref{BC_anis} we conclude easily by letting $\varepsilon\to 0$.
\end{proof}

Now we extend Lemma \ref{lemma_integr1} to a generic cut-off function in $\mathbb{R}^n$. Here, the convexity of $\Sigma$ plays a crucial role.

\begin{lemma}\label{lemma_cut-off}
	Let $v$ be given by \eqref{def_v}, let $V$ be as in \eqref{plus}, and $W$ as in \eqref{def_W}. Consider a non-negative cut-off function $\eta\in C_c^\infty(\R^n)$. Then
		 \begin{equation}\label{BC_integral_cut-off}
	 \begin{aligned}
	 \int_{\Sigma}&\left(2v^\gamma S^2(W)+\gamma(\gamma-1)p(p-1)v^{\gamma-2}V^2(\nabla v)+\gamma(2p-1)v^{\gamma-1}V(\nabla v)\Delta^H_p v\right) \eta \\\geq&-\int_{\Sigma}\eta_{j}\bigl(v^\gamma S_{ij}^2(W)V_{\xi_i}(\nabla v)+\gamma(p-1)v^{\gamma-1}V(\nabla v)V_{\xi_j}(\nabla v)\bigr) \, .
	 \end{aligned}
	 \end{equation}
\end{lemma}

\begin{proof}
As in the proof of Proposition \ref{Lemma_Caccioppoli}, this proof requires a regularization argument considering the solutions of the approximating problems
\begin{equation*}
\begin{cases}
\diver(a^\ell(\nabla v_{k,\ell})) =f(v,\nabla v) & \text{ in } \Sigma_k
\\
a^\ell(\nabla v_{k,\ell})\cdot\nu=0 & \text{ on } \partial\Sigma_k \,,
\end{cases}
\end{equation*}
where $a^\ell$ are defined as in \eqref{ajz} and $f(v,\nabla v)$ is given by \eqref{f_v_nablav}. Note that, since $v \in C^{1,\theta}_{\rm loc}(\Sigma \setminus \{\mathcal O\})$, 
the functions $v_{k,\ell}$ are of class $C^{2,\theta}_{\rm loc}$ in $\overline\Sigma_k\setminus \{\mathcal O\}$, and this allows one to perform all the desired computations on the functions $v_{k,\ell}$, and then let $\ell$ and $k$ to infinity. 
Since this approximation argument is very similar to the one in the proof of Proposition \ref{Lemma_Caccioppoli},
to simplify the notation and emphasize the main ideas we shall work directly with $v$, assuming that $v$ is of class $C^{2,\theta}_{\rm loc}$ in $\overline\Sigma\setminus \{\mathcal O\}$ in order to justify all the computations.

Set
\begin{equation}\label{F}
F= 2v^\gamma S^2(W)+\gamma(\gamma-1)p(p-1)v^{\gamma-2}V^2(\nabla v)+\gamma(2p-1)v^{\gamma-1}V(\nabla v)\Delta^H_p v
\end{equation}
and $L=(L_1,\ldots,L_n)$ with
$$
L_j = v^\gamma S_{ij}^2(W)V_{\xi_i}(\nabla v)+\gamma(p-1)v^{\gamma-1}V(\nabla v)V_{\xi_j}(\nabla v)
$$
for $j=1,\ldots , n$. 
Then we apply Lemma \ref{lemma_integr1} with $\varphi=\eta \zeta_\delta$, where $\eta \in C^\infty _c(\R^n)$ is a cut-off function as in the statement, and $\zeta_\delta \in C^\infty_c(\Sigma)$ is a cut-off function of the distance from $\partial \Sigma$ that converges to $1$ inside $\Sigma$ as $\delta\to 0$. In this way, as in the proof of \eqref{11.16}, letting $\delta \to 0$ the term involving $\nabla\zeta_\delta$
gives rise to a boundary term: more precisely, we obtain
\begin{equation}\label{BC_integral_3}
\int_{\Sigma } F \eta  = - \int_{\Sigma }\nabla \eta \cdot L + \int_{\partial \Sigma} \eta L \cdot \nu d \sigma \,.
\end{equation}
Now, to conclude the proof, we need to show that the last integral in \eqref{BC_integral_3} is non-negative; indeed, for $x \in \partial\Sigma\setminus \{\mathcal O\}$,  by using the explicit expression of $L$ and of $S^2_{ij}(W)$ we get
\begin{equation} \label{11febr}
\begin{aligned}
L&(x) \cdot \nu(x)= \\
 &v^\gamma(x) a(\nabla v(x))\cdot\nu(x)\left[{\tr}(W)(x)+\gamma(p-1)v^{-1}(x)V(\nabla v(x))\right] \\
 &-v^\gamma(x)\partial_i (a_j(\nabla v(x))) a_i(\nabla v(x))\nu_\ell(x)\, ,
\end{aligned}
\end{equation}
where we used that $w_{ji} (x)= \partial_i a_j(\nabla v(x))$ and $V_{\xi_i} = a_i$.  

We notice now that $\partial_i\nu_\ell (x)$ is the second fundamental form of $\partial\Sigma$ at $x$, which is non-negative definite by the convexity of $\Sigma$. Hence\
\begin{equation}\label{convexity_bis}
\partial_i\nu_\ell (x)a_j(\nabla v(x))a_i(\nabla v(x))\geq 0.
\end{equation}
%
From \eqref{11febr} and \eqref{convexity_bis} we get
\begin{equation*}
\begin{aligned}
L(x) \cdot \nu(x)&\geq v^\gamma(x)a(\nabla v(x))\cdot\nu(y)\left[{\tr}(W)(x)+\gamma(p-1)v^{-1}(x)V(\nabla v(x))\right] \\
&-v^\gamma(x)\nabla (a(\nabla v(x))\cdot\nu(y))\cdot a(\nabla v(x)).
\end{aligned}
\end{equation*}
Now, since $a(\nabla v)\cdot\nu=0$ on $\partial\Sigma$, the first term on the right-hand side vanishes. Moreover, since the tangential derivative of $a(\nabla v)\cdot\nu$ vanishes on $\partial\Sigma$ and $a(\nabla v)$ is a tangential vector-field, also the second term vanishes.
This proves that  $L\cdot \nu \geq 0$ on $\partial\Sigma\setminus \{\mathcal O\}$, that together with \eqref{BC_integral_3} (recall that $\eta\geq 0$)
concludes the proof.
\end{proof}

\begin{proposition}\label{mattinata}
	Let $v$ be given by \eqref{def_v}, let $V$ be as in \eqref{plus}, and $W$ as in \eqref{def_W}.  Then
	\begin{equation}\label{12.21}
	\begin{aligned}
	\int_{\Sigma}\left(2v^\gamma S^2(W)+\gamma(\gamma-1)p(p-1)v^{\gamma-2}V^2(\nabla v)+\gamma(2p-1)v^{\gamma-1}V(\nabla v)\Delta^H_p v\right) \geq 0 
	\end{aligned}
	\end{equation}
	for any $\gamma<-\frac{n(p-1)}{p} $.
\end{proposition}

\begin{proof} From \eqref{eq_v_bis_anis}, \eqref{v_asympt}, and \eqref{nabla_v_asympt} we know that $|\Delta^H_p v|\leq C$ in $\Sigma$, and from Newton's inequality \eqref{newtonIneq} we also have $|S^2(W)|\leq C$ (recall that ${\tr}(W)=\Delta^H_p v$).

Now, let $\eta$ be a non-negative radial cut-off function such that $\eta=1$ in $B_R$, $\eta=0$ outside $B_{2R}$, and $|\nabla\eta|\leq \frac{2}{R}$.
Thanks to \eqref{v_asympt} and \eqref{nabla_v_asympt}, we can take the limit as $R\to \infty$ in the left-hand side of \eqref{BC_integral_cut-off} to obtain the left-hand side of \eqref{12.21}. Hence, in order to prove \eqref{12.21} it is enough to show that
	\begin{equation}\label{va_a_0}
	\lim_{R\rightarrow\infty}\int_{E_R}\eta_{j}\bigl(v^\gamma S_{ij}^2(W)V_{\xi_i}(\nabla v)+\gamma(p-1)v^{\gamma-1}V(\nabla v)V_{\xi_j}(\nabla v)\bigr)=0 \, ,
	\end{equation}
	where we set for simplicity
	$$
	E_R:=\Sigma\cap(B_{2R}\setminus B_R)\, 
	$$
	Since $|S^2_{ij}(W)|\leq |W|$, using  Holder's inequality we get 
	$$
	\left|\int_{E_R}\eta_{j}v^\gamma S_{ij}^2(W)V_{\xi_i}(\nabla v)\right|\leq \dfrac{c(n)}{R} \|W\|_{L^2(E_R)}\left(\int_{E_R} v^{2\gamma}|\nabla V(\nabla v)|^2\right)^{\frac{1}{2}}\, .
	$$
	Observe that \eqref{caccioppoli_asymptotic} yields
	$$
	\|W\|^2_{L^2(E_R)}\leq CR^{n}\, .
	$$
	Also, from \eqref{v_asympt} and \eqref{nabla_v_asympt} we have
	$$
	\int_{E_R} v^{2\gamma}|\nabla V(\nabla v)|^2\leq C R^{\frac{2\gamma p}{p-1}+n+2} \, .
	$$
	Hence, since by assumption $\gamma<-\frac{n(p-1)}{p}$, this proves that
	$$
	\lim_{R\rightarrow\infty}\int_{E_R}\eta_{j}v^\gamma S_{ij}^2(W)V_{\xi_i}(\nabla v)=0\,.
	$$
	Analogously,  using \eqref{v_asympt} and \eqref{nabla_v_asympt}, the second term in \eqref{va_a_0} can be bounded as
	\begin{equation}
\begin{aligned}
\left|\int_{E_R}\eta_{j}v^{\gamma-1}V(\nabla v)V_{\xi_j}(\nabla v)\right|\leq C R^{\frac{p\gamma}{p-1}+n}\,,
\end{aligned}
	\end{equation}
which also goes to zero as $R\to\infty$ since $\gamma<-\frac{n(p-1)}{p}$. This proves \eqref{va_a_0} and hence \eqref{12.21}.
\end{proof}

\subsection{Conclusion} We multiply \eqref{eq_v_bis_anis} by $v^{-n}$ and integrate over $\Sigma$. By using the divergence theorem, the boundary condition in \eqref{eq_v_bis_anis}, and the decay estimates \eqref{v_asympt} and \eqref{nabla_v_asympt}, we get
\begin{equation}\label{=0}
\left(\dfrac{p}{n-p}\right)^{p-1}\int_{\Sigma}v^{-n-1}-\dfrac{n}{p}\int_{\Sigma}v^{-n-1} H^p(\nabla v)=0\, .
\end{equation}
Now we use Newton's inequality applied to $W$ in \eqref{12.21}. More precisely, since $\tr (W) = \Delta_p^H v$, we have
\begin{equation}\label{Newton_W}
2S^2(W)\leq \dfrac{n-1}{n} (\Delta^H_p v)^2 \, ,
\end{equation}
and from \eqref{12.21} we obtain 
\begin{equation}\label{17.42}
\begin{aligned}
\int_{\Sigma}\left(\dfrac{n-1}{n}v^{\gamma} (\Delta^H_p v)^2+\gamma(\gamma-1)p(p-1)v^{\gamma-2}V^2(\nabla v)+\gamma(2p-1)v^{\gamma-1}V(\nabla v)\Delta^H_p v\right) \geq 0 
\end{aligned}
\end{equation}
for any $\gamma<-\frac{n(p-1)}{p}$. Since $p<n$ we can choose $\gamma=1-n$ in \eqref{17.42}, and using \eqref{eq_v_bis_anis}, \eqref{f_v_nablav},
and \eqref{plus}, we obtain 
\begin{equation}\label{quasi_equality}
\left(\dfrac{p}{n-p}\right)^{p-1}\int_{\Sigma}v^{-n-1} -\dfrac{n}{p}\int_{\Sigma}v^{-n-1}H^p(\nabla v)\geq 0\, .
\end{equation}
Recalling \eqref{=0}, this implies that the equality case must hold in \eqref{quasi_equality}. Hence the equality case must hold in \eqref{Newton_W} a.e., which implies that 
\begin{equation}\label{=_anis}
W(x)=\lambda(x){\rm Id} \quad \text{ for a.e. }\, x\in\Sigma\, ,
\end{equation}
for some function $\lambda:\Sigma\to \R$, where $I$ is the identity matrix. 

Now we show that the function $\lambda$ is constant. Since 
$$
\lambda(x)=\frac{1}{n} \tr (W) =  \frac{1}{n}\Delta^H_p v(x)=\frac{1}{n} f(v,\nabla v)
$$ 
(see \eqref{eq_v_bis_anis}), and since $v\in C^{1,\theta}_{\rm loc}(\Sigma)$, we get that $\lambda\in C^{0,\theta}_{\rm loc}(\Sigma)$. 
Moreover, elliptic regularity theory yields that $v\in C^{2,\theta}_{\rm loc}(\Sigma\cap\{\nabla v\neq 0\})$, which implies that $\lambda\in C^{1,\theta}_{\rm loc}(\Sigma\cap\{\nabla v\neq 0\})$. From \eqref{=_anis} we have that 
\begin{equation}\label{=_anis2}
\partial_i(a_j(\nabla v(x)))=\lambda(x)\delta_{ij}
\end{equation}
for $i,j\in \{1,\ldots,n\}$, which implies that $a(\nabla v)\in C^{2,\theta}_{\rm loc}(\Sigma\cap\{\nabla v\neq 0\})$. 
 
Then, given $i \in  \{1,\ldots,n\}$, choosing $j \neq i$ and using \eqref{=_anis2} we obtain
$$
\partial_i\lambda(x) = \partial_i\bigl(\partial_j(a_j(\nabla v(x)))\bigr)=\partial_j\bigl(\partial_i(a_j(\nabla v(x)))\bigr)=0
$$
for any $x \in \Sigma\cap\{\nabla v\neq 0\}$, which implies that $\lambda$ is constant on each connected component of $\Sigma\cap\{\nabla v\neq 0\}$. Since $\lambda$ is continuous in $\Sigma$ and $\{\nabla v= 0\}$ has no interior points (this follows easily from \eqref{eq_v_bis_anis}), we deduce that $\lambda$ is constant. In particular, recalling \eqref{=_anis}, we get
$$
\nabla[a(\nabla v(x))]=W(x)=\lambda \,I \qquad \text{in }\Sigma\,.
$$ 
Hence $a(\nabla v(x))=\lambda (x-x_0)$ for some $x_0\in\overline{\Sigma}$, and from the boundary condition in \eqref{eq_v_bis_anis} we obtain that $x_0\in\partial\Sigma$.  This implies that $v(x)=c_1+c_2 H_0(x-x_0)^{\frac{p}{p-1}}$,
or equivalently (recalling \eqref{def_v}) $u(x) = U_{\mu,x_0}^H (x)$ for some $\mu>0$.  
Finally, it is clear that:\\
- if $\Sigma = \mathbb{R}^n$ and $x_0$ may be a generic point in $\mathbb{R}^n$;\\
- if $k\in\{1,\dots,n-1\}$ then $x_0\in\mathbb{R}^k\times\mathcal{\{\mathcal O\}}$; \\
- if $k=0$ then $x_0=\mathcal{O}$.

This completes the proof of Theorem \ref{thm_main}.

\appendix

\section{Sharp anisotropic Sobolev inequalities with weight in convex cones}\label{appendix_sobolev}
In this appendix we prove a sharp version of the anisotropic Sobolev inequality in cones by suitably adapting the optimal transportation proof of the Sobolev inequality in \cite[Theorem 2]{CNV}.
As we shall see, the proof not only applies to the case of arbitrary norms, but it also allows us to cover a large class of weights. In particular, our result extends the weighted isoperimetric inequalities from \cite[Theorem 1.3]{CabreRosSerra} to the full Sobolev range $p \in (1,n)$ (note that the case $p=1$ can be recovered letting $p\to 1^+$).

 \begin{theorem} \label{thm_sob_ineq}
Let $p\in(1,n)$. Let $\Sigma$ be a convex cone and $H$ a norm in $\mathbb{R}^n$. Let $w \in C^0(\overline \Sigma)$ be  positive in $\Sigma$, homogeneous of degree $a \geq 0$, and such that $w^{1/a}$ is concave in case $a>0$. Then for any $f\in \mathcal{D}^{1,p}(\Sigma)$ we have 
\begin{equation} \label{sobolev_thm}
\biggl(\int_{\Sigma}|f(x)|^{\beta}w(x)\, dx\biggr)^{p/\beta} \leq C_\Sigma(n,p,a,H,w) \int_{\Sigma}H^p(\nabla f(x))\, w (x) \,dx
\end{equation}
where 
\begin{equation} \label{beta_thm}
\beta=\frac{p(n+a)}{n+a-p} \,.
\end{equation}
Moreover, inequality \eqref{sobolev_thm} is sharp and the equality is attained if and only if $f=U_{\lambda,x_0}^{H,a} $, where
\begin{equation} \label{talentiane_pa_H}
U_{\lambda,x_0}^{H,a} (x) := \left( \frac{\lambda^{\frac{1}{p-1}}c(n,p,a,H,w) }{\lambda^\frac{p}{p-1} + H_0(x-x_0)^\frac{p}{p-1} } \right)^{\frac{n+a-p}{p}}
\end{equation}
with $\lambda >0$ , and $H_0(\zeta):=\sup_{H(\xi)=1}\zeta\cdot \xi$ is the dual norm of $H$.

Furthermore, writing $\Sigma=\mathbb{R}^k\times\mathcal{C}$ with $k\in\{0,\dots,n\}$ and with $\mathcal C \subset \mathbb{R}^{n-k}$ a convex cone that does not contain a line, then:
\begin{itemize}
	\item[$(i)$] if $k=n$ then $\Sigma = \mathbb{R}^n$ and $x_0$ may be a generic point in $\mathbb{R}^n$;
	\item[$(ii)$] if $k\in\{1,\dots,n-1\}$ then $x_0\in\mathbb{R}^k\times\mathcal{\{\mathcal O\}}$;
	\item[$(iii)$] if $k=0$ then $x_0=\mathcal{O}$.
\end{itemize} 
 
\end{theorem}

\begin{proof}
We aim at proving that for any nonnegative $f,g \in L^{\beta}( \Sigma)$ with $\|f\|_{L^{\beta}(\Sigma)}=\|g\|_{L^{\beta}(\Sigma)}$ and such that $\nabla f \in L^p(\Sigma)$, we have that
\begin{equation}\label{20.01}
\int_{\Sigma}g^\gamma w\,dx\leq \dfrac{\gamma}{n+a}\left(\int_{\Sigma} H^p(\nabla f)\,w\,dx \right)^{1/p}\left(\int_{\Sigma}H_0^{p'}g^\beta w\,dx\right)^{1/p'}\, ,
\end{equation} 
with equality if $f=g=U_{\lambda,x_0}^{H,a} $. The value of $\gamma$ will be specified later. As shown in \cite{CNV}, inequality \eqref{20.01} implies the Sobolev inequality \eqref{sobolev_thm}.

Let $F$ and $G$ be probability densities on $\Sigma$ and let $T:\Sigma\rightarrow\Sigma$ be the optimal transport map (see e.g. \cite{Villani}).\footnote{ 
As explained in \cite{FigIndrei} (see also \cite{FigMagPrat}), the argument that follows can be made rigorous using the fine properties of $BV$ functions (we note that $T$ belongs to $BV$, being the gradient of a convex function).
However, to emphasize the main ideas, we shall write the whole argument when $T:\Sigma\to \Sigma$ is a $C^1$ diffeomorphism, and we invite the interested reader to  
look at the proof of \cite[Theorem 2.2]{FigIndrei} to understand how to adapt the argument using only that $T \in BV_{\rm loc}(\Sigma;\Sigma)$.

Alternatively, arguing by approximation, one can assume that $w$ is strictly positive in $\overline\Sigma\setminus \{0\}$, and that $f$ and $g$ are both strictly positive and smooth inside $\overline\Sigma.$
Then, if $T:\Sigma\to \Sigma$ denotes the optimal transport map from $f^\beta w$ to $g^\beta w$, \cite[Theorem 1 and Remark 4]{CEF} ensure that $T:\Sigma\to \Sigma$ is a diffeomorphism. This allows one to perform the proof of \eqref{20.01} avoiding the use of the fine properties of $BV$ functions.}
It is well known that, by the transport condition $T_\# F=G$, one has
$$
|\det(DT)|=\dfrac{F}{G\circ T}
$$
(see for instance \cite[Section 3]{DepF}).
Then, if we choose 
$$
F=f^\beta w \quad \text{and} \quad G=g^\beta w \,,
$$  
the Jacobian equation for $T$ becomes
$$
|\det(DT)|\,\dfrac{w\circ T}{w}=\dfrac{f^\beta}{g^\beta\circ T}\, .
$$
We observe that, since
$$
T_{\#}(f^\beta w)=g^\beta w\, ,
$$
then for any $0<\gamma<\beta$ we have 
\begin{equation}\label{19.10}
\int_{\Sigma}g^\gamma w\,dx= \int_{\Sigma}(g^{\gamma-\beta}\circ T ) f^\beta w\,dx 
= \int_{\Sigma} \left[|\det(DT)|\,\dfrac{w\circ T}{w}\right]^{\frac{\beta-\gamma}{\beta}}f^\gamma w\,dx\, .
\end{equation}
We choose $\gamma$ such that 
$$
\dfrac{\beta-\gamma}{\beta}=\dfrac{1}{n+a}\, \quad  \text{i.e.} \quad \gamma=\frac{p(n+a-1)}{n+a-p} \,.
$$
Since $T=\nabla \varphi$ for some convex function $\varphi$, then $DT$ is  symmetric and  nonnegative definite.
In particular $\det(DT)\geq 0$,
and it follows from  Young and the arithmetic-geometric inequalities that
$$
\begin{aligned}
\left[|\det(DT)|\,\dfrac{w\circ T}{w}\right]^{\frac{1}{n+a}}&\leq \dfrac{n}{n+a}\det(DT)^{1/n}+\dfrac{a}{n+a}\left(\dfrac{w\circ T}{w}\right)^{1/a} \\
& \leq \dfrac{1}{n+a}\left[\diver(T)+a\left(\dfrac{w\circ T}{w}\right)^{1/a}\right] \, .  
\end{aligned}
$$
Also, from the concavity of $w^{1/a}$ we have that 
$$
a\left(\dfrac{w\circ T}{w}\right)^{1/a} \leq \frac{\nabla w \cdot T}{w}
$$
(see \cite[Lemma 5.1]{CabreRosSerra}), hence
\begin{equation} \label{51}
\left[|\det(DT)|\,\dfrac{w\circ T}{w}\right]^{\frac{1}{n+a}} \leq  \dfrac{1}{n+a}\left(\diver(T)+ \frac{\nabla w \cdot T}{w} \right) \, .  
\end{equation}
(If $a=0$ then $w$ is just constant and \eqref{51} corresponds to the arithmetic-geometric inequality.)
Noticing that
$$
\diver(T)+ \frac{\nabla w \cdot T}{w} = \frac{1}{w} \diver(T w)\,,
$$
combining \eqref{19.10} and \eqref{51} we have   
$$
\begin{aligned}
\int_{\Sigma}g^\gamma w\,dx &\leq\dfrac{1}{n+a}\int_{\Sigma}\diver(T w) f^{\gamma}\,dx \\
&=-\dfrac{\gamma}{n+a}\int_{\Sigma}w f^{\gamma-1}T\cdot\nabla f \,dx+ \dfrac{1}{n+a}\int_{\partial\Sigma}w f^\gamma T\cdot\nu \,d\sigma\, .
\end{aligned}
$$
Here we notice that, since $T(x)\in\overline\Sigma$ for any $x \in \overline\Sigma$, the convexity of $\Sigma$ implies that $T \cdot \nu \leq 0$ on $\partial \Sigma$. Thus we obtain 
$$
	\int_{\Sigma}g^\gamma w \,dx\leq -\dfrac{\gamma}{n+a}\int_{\Sigma} f^{\gamma-1}T\cdot\nabla f\,w \,dx \leq \dfrac{\gamma}{n+a}\int_{\Sigma} f^{\gamma-1}H_0(T)H(\nabla f)\,w\,dx\, ,
$$
where the last inequality follows from the definition of the dual norm $H_0$.
%
%
%
Finally,  setting $p'=\frac{p}{p-1}$,
it follows by Holder's inequality that
$$
\begin{aligned}
\int_{\Sigma} f^{\gamma-1}H_0(T)H(\nabla f) \,w\,dx&
\leq\left(\int_{\Sigma} f^{p(\gamma-1)-\frac{p\beta}{p'}}H^p(\nabla f)\,w\,dx\right)^{1/p}\left(\int_{\Sigma} H_0^{p'}(T)\,f^{\beta}w\,dx \right)^{1/p'}\\
&=\left(\int_{\Sigma} H^p(\nabla f)\,w\,dx\right)^{1/p}\left(\int_{\Sigma} H_0^{p'}g^{\beta}w\,dx\right)^{1/p'} \, ,
\end{aligned}
$$
where we used the transport condition $T_\#(f^\beta w)=g^\beta w$ and the identity
$$
\gamma-1-\frac{\beta}{p'}=0\, .
$$
Hence, by this chain of inequalities we get \eqref{20.01}.

In order to prove the sharpness of our Sobolev inequality we choose $f=g=U_{1,\mathcal O}^{H,a} $. In this particular case the transport map reduces to the identity map $T(x)=\nabla\varphi(x) = x$ and $\det(DT)=1$. Also the homogeneity of $w$ implies that $\nabla w\cdot x=a\,w$. This implies that all the inequalities in the previous computations become equalities and we obtain \eqref{sobolev_thm}.

Finally, to prove the characterization of the minimizers one can argue as in \cite[Appendix A]{FigMagPrat} and \cite[Section 4]{CNV}.
More precisely, choose $g=U_{1,\mathcal O}^{H,a}$ and let $f$ be a minimizer. As noticed in the proof of \cite[Theorem 5]{CNV},
one can assume that $f \geq 0$.

First one shows that the support of $f$ is indecomposable (this is a measure-theoretic notion of the concept that $\{f>0\}$ is connected, see \cite[Appendix A]{FigMagPrat} for a definition and more details).
Indeed, otherwise one could write $f=f_1+f_2$ with 
$$
\int_{\Sigma}H^p(\nabla f) w (x) dx=\int_{\Sigma}H^p(\nabla f_1) w (x) dx+\int_{\Sigma}H^p(\nabla f_2) w (x) dx
$$
and then by applying \eqref{sobolev_thm} and the fact that $f$ is a minimizer, we would get
$$
\biggl(\int_{\Sigma}f^{\beta}w(x) dx\biggr)^{p/\beta}\geq \biggl(\int_{\Sigma}f_1^{\beta}w(x) dx\biggr)^{p/\beta}+\biggl(\int_{\Sigma}f_2^{\beta}w(x) dx\biggr)^{p/\beta}.
$$
Since
$$
\int_{\Sigma}f^{\beta}w(x) dx=\int_{\Sigma}f_1^{\beta}w(x) dx+\int_{\Sigma}f_2^{\beta}w(x) dx
$$
(because $f_1$ and $f_2$ have disjoint support), by concavity of the function $t\mapsto t^{p/\beta}$ we conclude that either $f_1$ or $f_2$ vanishes.

Once this is proved, one can then argue as in the proof of \cite[Proposition 6]{CNV} to deduce (from the fact that all the inequalities in the proof given above much be equalities) that $T$ must be of the form $T(x)=\lambda (x-x_0)$ for some $\lambda >0$ and $x_0 \in\Sigma$, from which the result follows easily.
Finally, properties (i)-(ii)-(iii) on the location of $x_0$ follow for instance from the fact that $T$ has to map $\Sigma$ onto $\Sigma$.
\end{proof}

\end{document}